\documentclass[10pt,a4paper]{article}
\usepackage{fullpage,mathrsfs,graphicx,framed,color,amssymb,amsmath,amsthm}
\usepackage[boxed, noline, ruled, linesnumbered]{algorithm2e}
\usepackage{epsfig, graphicx}
\usepackage{latexsym,amsfonts,amsbsy,amssymb}
\usepackage{amsmath,amsthm}
\usepackage{enumerate}
\makeatletter %'@' is now a normal "letter" for TeX

\@addtoreset{equation}{section}
\makeatother %'@' is restored as a "non-letter" character for TeX
\allowdisplaybreaks % For long formula
\newtheorem{theorem}{Theorem}[section]
\newtheorem{lemma}{Lemma}[section]
\newtheorem{corollary}{Corollary}[section]
\newtheorem{remark}{Remark}[section]

\newtheorem{proposition}{Proposition}[section]

\newcommand{\comm}[1]{{\color{black}#1}}
\newcommand{\revise}[1]{{\color{black}#1}}
\begin{document}
\title{\revise{A  Parallel Augmented Subspace Method for Eigenvalue
Problems}\footnote{This work was supported in part
by the National Key Research and Development Program of China (2019YFA0709601), Science Challenge Project (No. TZ2019002),
National Natural Science Foundations of China (NSFC 11771434, 11801021, 91730302, 91630201), the National Center for Mathematics and Interdisciplinary Science, CAS.}}
\author{Fei Xu\footnote{Beijing Institute for Scientific and Engineering
Computing, College of applied sciences, Beijing University of Technology, Beijing, 100124, China
(xufei@lsec.cc.ac.cn).},\ \ \
Hehu Xie\footnote{ICMSEC, LSEC, NCMIS,
Academy of Mathematics and Systems Science, Chinese Academy of
Sciences, Beijing 100190, China, and School of Mathematical
Sciences, University of Chinese Academy of Sciences, Beijing,
100049, China ({\tt hhxie@lsec.cc.ac.cn}).} \ \ \ and \ \
Ning Zhang\footnote{Institute of Electrical Engineering, Chinese Academy of Sciences, No.6,
Beiertiao, Zhongguancun, Haidian, Beijing 100190, China (zhangning@\allowbreak mail.iee.ac.cn).}
} \date{} \maketitle
%----------------------------------------------------------------------------------------
\begin{abstract}
A type of parallel augmented subspace scheme for eigenvalue problems is proposed
by using coarse space in the multigrid method.
With the help of coarse space in multigrid method, solving the eigenvalue problem
in the finest space is decomposed into solving the standard linear boundary value
problems and very low dimensional eigenvalue problems.
The computational efficiency can be improved since there is no
direct eigenvalue solving in the finest space and the multigrid method can act as the
solver for the deduced linear boundary value problems.
Furthermore, for different eigenvalues, the corresponding boundary value problem
and low dimensional eigenvalue problem can be solved in the parallel way
since they are independent of each other and there exists no data exchanging.
This property means that we do not need to do the orthogonalization in the highest
dimensional spaces. This is the main aim of this paper since avoiding orthogonalization
can improve the scalability of the proposed numerical method. Some numerical examples are
provided to validate the proposed parallel augmented subspace method.


\end{abstract}

%-----------------------------------------------------------------------------------------
\section{Introduction}
Solving large scale eigenvalue problems is one of  fundamental problems in modern science
and engineering society. It is always a very difficult task to solve high-dimensional
eigenvalue problems which come from practical physical and chemical sciences.
Compared with boundary value problems, there are less efficient numerical
methods for solving eigenvalue problems with optimal complexity.
The large scale eigenvalue problems pose significant challenges for scientific computing.
In order to solve these large scale sparse eigenvalue problems, Krylov subspace type methods
(Implicitly Restarted Lanczos/Arnoldi Method (IRLM/IRAM) \cite{Sorensen}),
the Preconditioned INVerse ITeration (PINVIT) method \cite{PINVIT,BramblePasciakKnyazev,Knyazev},
the Locally Optimal Block Preconditioned Conjugate Gradient (LOBPCG) method \cite{Knyazev_Lobpcg,KnyazevNeymeyr},
and the Jacobi-Davidson-type techniques \cite{Bai} have been developed.
All these popular methods include the orthogonalization step which is a bottleneck for designing
\revise{efficient parallel schemes for determining relatively many eigenpairs}.
Recently, a type of multilevel correction method is proposed for solving
eigenvalue problems in \cite{LinXie_MultiLevel,Xie_IMA,Xie_JCP}. In this multilevel
correction scheme, there exists an augmented subspace which is constructed with the help of coarse space
from the multigrid method. The application of this augmented subspace leads to that
the solution of eigenvalue problem on the final level of mesh can be reduced to a series
of solutions of boundary value problems on the multilevel meshes and a series of
solutions of the eigenvalue problem on the low dimensional augmented subspace.
The multilevel correction method gives a way to construct the multigrid method
for eigenvalue problems.

It is well known that the multigrid method \cite{BrennerScott,Shaidurov,Xu} provides
an optimal numerical method for linear elliptic boundary value problems.
The error bounds of the approximate solution obtained from these efficient
numerical algorithms are comparable to the theoretical bounds determined
by the finite element discretization, while the amount of computational
work involved is only proportional to the number of unknowns
in the discretized equations. For more details of the multigrid method,
please refer to \cite{BankDupont,Bramble,BramblePasciak,BrambleZhang,BrennerScott,
Hackbusch,Hackbusch_Book,McCormick,ScottZhang,Shaidurov,Xu,Xu_Two_Grid} and the references cited therein.

This paper aims to design a type of  parallel method for eigenvalue problems
with the help of the coarse space from the multigrid method.
It is well known that there exist many work considering the applications of multigrid method
for eigenvalue problems. For example, there have existed applications of the multigrid method
to the PINVIT and LOBPCG methods. But, in these applications, the multigrid method only acts
as the precondition for the included linear equations. This means that the multigrid method
only improves the efficiency of the inner iteration and does not change the outer iteration.
Unfortunately,  in these state-of-the-art, the applications of multigrid method
do not deduce a new eigensolver. The idea of designing the  parallel
augmented subspace method for eigenvalue problems is based on the combination of
the multilevel correction method \cite{LinXie_MultiLevel,Xie_IMA,Xie_JCP,XieZhangOwhadi} and parallel
computing technique. With the help of coarse space in multigrid method, the eigenvalue
problem solving is transformed into a series of solutions of the corresponding linear
boundary value problems on the sequence of finite element spaces and eigenvalue problems
on a very low dimensional augmented space. Further, in order to improve the parallel scalability,
%the correction process for different eigenpair is executed independently.
\revise{the approximate eigenpairs are computed independently by a correction process}.
This property means there is no orthogonalization in the highest dimensional space which account
for a large portion of wall time in the parallel computation.

An outline of the paper goes as follows. In Section 2, we introduce the
finite element method for the eigenvalue problem and the corresponding basic
error estimates. A type of parallel augmented subspace method
for solving the eigenvalue problem by finite element method
is given in Section 3,  and the corresponding computational
work estimate are given in Section 4.
In Section 5, four numerical examples are presented to validate our
theoretical analysis. Some concluding remarks are given in the last section.

%===================================================================================================
\section{Finite element method for eigenvalue problem}\label{sec;preliminary}
This section is devoted to introducing some notation and the standard finite element
method for the eigenvalue problem. In this paper, we shall use the standard notation
for Sobolev spaces $W^{s,p}(\Omega)$ and their
associated norms and semi-norms (cf. \cite{Adams}). For $p=2$, we denote
$H^s(\Omega)=W^{s,2}(\Omega)$ and
$H_0^1(\Omega)=\{v\in H^1(\Omega):\ v|_{\partial\Omega}=0\}$,
where $v|_{\Omega}=0$ is in the sense of trace,
$\|\cdot\|_{s,\Omega}=\|\cdot\|_{s,2,\Omega}$.
In some places, $\|\cdot\|_{s,2,\Omega}$ should be viewed as piecewise
defined if it is necessary.
The letter $C$ (with or without subscripts) denotes a generic
positive constant which may be different at
 its different occurrences through the paper.

For simplicity, we consider the following model problem to illustrate the main idea:
Find $(\lambda, u)$ such that
\begin{equation}\label{LaplaceEigenProblem}
\left\{
\begin{array}{rcl}
-\nabla\cdot(\mathcal{A}\nabla u)&=&\lambda u, \quad {\rm in} \  \Omega,\\
u&=&0, \ \ \quad {\rm on}\  \partial\Omega,
\end{array}
\right.
\end{equation}
where $\mathcal{A}$ is a symmetric and positive definite matrix with suitable
regularity, $\Omega\subset\mathcal{R}^d\ (d=2,3)$ is a bounded domain with
Lipschitz boundary $\partial\Omega$.

In order to use the finite element method to solve
the eigenvalue problem (\ref{LaplaceEigenProblem}), we need to define
the corresponding variational form as follows:
Find $(\lambda, u )\in \mathcal{R}\times V$ such that $a(u,u)=1$ and
\begin{eqnarray}\label{weak_eigenvalue_problem}
a(u,v)=\lambda b(u,v),\quad \forall v\in V,
\end{eqnarray}
where $V:=H_0^1(\Omega)$ and
\begin{equation}\label{inner_product_a_b}
a(u,v)=\int_{\Omega}\mathcal{A}\nabla u\cdot\nabla v d\Omega,
 \ \ \ \  \ \ b(u,v) = \int_{\Omega}uv d\Omega.
\end{equation}
The norms $\|\cdot\|_a$ and $\|\cdot\|_b$ are defined by
\begin{eqnarray*}
\|v\|_a=\sqrt{a(v,v)}\ \ \ \ \ {\rm and}\ \ \ \ \ \|v\|_b=\sqrt{b(v,v)}.
\end{eqnarray*}
It is well known that the eigenvalue problem (\ref{weak_eigenvalue_problem})
has an eigenvalue sequence $\{\lambda_j \}$ (cf. \cite{BabuskaOsborn_1989,Chatelin}):
$$0<\lambda_1\leq \lambda_2\leq\cdots\leq\lambda_k\leq\cdots,\ \ \
\lim_{k\rightarrow\infty}\lambda_k=\infty,$$ and associated
eigenfunctions
$$u_1, u_2, \cdots, u_k, \cdots,$$
where $a(u_i,u_j)=\delta_{ij}$ ($\delta_{ij}$ denotes the Kronecker function).
In the sequence $\{\lambda_j\}$, the $\lambda_j$ are repeated according to their
geometric multiplicity.
For our analysis,  recall the following definition for the smallest eigenvalue
(see \cite{BabuskaOsborn_Book,Chatelin})
\begin{eqnarray}\label{Smallest_Eigenvalue}
\lambda_1 = \min_{0\neq w\in V}\frac{a(w,w)}{b(w,w)}.
\end{eqnarray}

%-----------------------------------------------------------------------------------------
Now, let us define the finite element approximations of the problem
(\ref{weak_eigenvalue_problem}). First we generate a shape-regular
 triangulation $\mathcal{T}_h$ of the computing domain $\Omega\subset \mathcal{R}^d\
(d=2,3)$ into triangles or rectangles for $d=2$ (tetrahedrons or
hexahedrons for $d=3$). The diameter of a cell $K\in\mathcal{T}_h$
is denoted by $h_K$ and the mesh size $h$ describes  \revise{the maximal diameter} of all cells
$K\in\mathcal{T}_h$.
Based on the mesh $\mathcal{T}_h$, we can construct a finite element space denoted by
 $V_h \subset V$. For simplicity, we set $V_h$ as the linear finite
 element space which is defined as follows
\begin{equation}\label{linear_fe_space}
  V_h = \big\{ v_h \in C(\Omega)\ \big|\ v_h|_{K} \in \mathcal{P}_1,
  \ \ \forall K \in \mathcal{T}_h\big\}\cap H^1_0(\Omega),
\end{equation}
where $\mathcal{P}_1$ denotes the linear function space.

The standard finite element scheme for eigenvalue
problem (\ref{weak_eigenvalue_problem}) is:
Find $(\bar{\lambda}_h, \bar{u}_h)\in \mathcal{R}\times V_h$
such that $a(\bar{u}_h,\bar{u}_h)=1$ and
\begin{eqnarray}\label{Weak_Eigenvalue_Discrete}
a(\bar{u}_h,v_h)=\bar{\lambda}_h b(\bar{u}_h,v_h),\quad\ \  \ \forall v_h\in V_h.
\end{eqnarray}
From \cite{BabuskaOsborn_1989,BabuskaOsborn_Book}, the  discrete eigenvalue
problem (\ref{Weak_Eigenvalue_Discrete}) has eigenvalues:
$$0<\bar{\lambda}_{1,h}\leq \bar{\lambda}_{2,h}\leq\cdots\leq \bar{\lambda}_{k,h}
\leq\cdots\leq \bar{\lambda}_{N_h,h},$$
and corresponding eigenfunctions
\begin{eqnarray}\label{Discrete_Eigenfunctions}
\bar{u}_{1,h}, \bar{u}_{2,h}, \cdots, \bar{u}_{k,h}, \cdots, \bar{u}_{N_h,h},
\end{eqnarray}
where $a(\bar{u}_{i,h},\bar{u}_{j,h})=\delta_{ij}$, $1\leq i,j\leq N_h$ ($N_h$ is
the dimension of the finite element space $V_h$).
From the min-max principle \cite{BabuskaOsborn_1989,BabuskaOsborn_Book},
\revise{the eigenvalues of (\ref{Weak_Eigenvalue_Discrete}) provide upper bounds for the first
$N_h$ eigenvalues of (\ref{weak_eigenvalue_problem})}
%we have the following upper bound result
\begin{eqnarray}\label{Uppero_Bound_Result}
\lambda_i \leq \bar\lambda_{i,h},\ \ \ \ 1\leq i\leq N_h.
\end{eqnarray}

%Let $M(\lambda_i)$ denote the eigenspace corresponding to the
%eigenvalue $\lambda_i$ which is defined by
%\begin{eqnarray}
%M(\lambda_i)=\big\{w\in H_0^1(\Omega): w\ {\rm is\ an\ eigenfunction\ of\
%(\ref{weak_eigenvalue_problem})\ corresponding\ to} \ \lambda_i\big\},
%\end{eqnarray}
%and define
%\begin{eqnarray}
%\delta_h(\lambda_i)=\sup_{\substack{w\in M(\lambda_i)\\ \|w\|_a=1 }}\inf_{v_h\in
%V_h}\|w-v_h\|_{a}.
%\end{eqnarray}
In order to measure the error of the finite element space to the desired function, we define the following notation
\begin{eqnarray}\label{Delta_V_h}
\delta(w,V_h) = \inf_{v_h\in V_h}\|w-v_h\|_a,\ \ \ {\rm for}\ w\in V.
\end{eqnarray}
In this paper, we also need the following quantity for error analysis:
\begin{eqnarray}
\eta(V_h)&=&\sup_{\substack{ f\in L^2(\Omega)\\ \|f\|_b=1}}\inf_{v_h\in V_h}\|Tf-v_h\|_{a},\label{eta_a_h_Def}
\end{eqnarray}
where $T:L^2(\Omega)\rightarrow V$ is defined as
\begin{equation}\label{laplace_source_operator}
\revise{a(Tf,v) = b(f,v), \ \ \ \ \  \forall v\in V\ \ {\rm for}\  f \in L^2(\Omega).}
\end{equation}

In order to understand the method more clearly, we state the error estimate for the
eigenpair approximation by the finite element method. For this aim,
we define the finite element projection $\mathcal P_h$ as follows
\begin{eqnarray}\label{Energy_Projection}
\revise{a(\mathcal P_h w, v_h) = a(w,v_h),\ \ \ \ \forall v_h\in V_h \ \ {\rm for}\  w\in V.}
\end{eqnarray}
It is obvious that
\begin{eqnarray}\label{Delta_V_h_P_h}
\revise{\|u-\mathcal P_h u\|_a = \inf_{w_h\in V_h}\|u-w_h\|_a = \delta(u,V_h),\ \ \ \forall u\in V.}
\end{eqnarray}

%Before stating error estimates of the subspace projection method, we introduce a lemma which
%comes from \cite{StrangFix}. For completeness, a proof is also provided here.
%\begin{lemma}(\cite[Lemma 6.4]{StrangFix})\label{Strang_Lemma}
%For any exact eigenpair $(\lambda,u)$ of (\ref{weak_eigenvalue_problem}), the following equality holds
%\begin{eqnarray*}\label{Strang_Equality}
%(\bar{\lambda}_{j,h}-\lambda)b(\mathcal P_hu, \bar u_{j,h})
%=\lambda b(u-\mathcal P_hu, \bar u_{j,h}),\ \ \
%j = 1, \cdots, N_h.
%\end{eqnarray*}
%\end{lemma}
%%-------------------------------------------------------------------------------------------------
%\begin{proof}
%Since $-\lambda b(\mathcal P_hu, \bar u_{j,h})$ appears on both sides, we only need to prove that
%\begin{eqnarray*}
%\bar{\lambda}_{j,h} b(\mathcal P_h u, \bar u_{j,h})=\lambda b(u, \bar u_{j,h}).
%\end{eqnarray*}
%From (\ref{weak_eigenvalue_problem}), (\ref{Weak_Eigenvalue_Discrete}) and (\ref{Energy_Projection}),
%the following equalities hold
%\begin{eqnarray*}
%\bar{\lambda}_{j,h}b(\mathcal P_hu, \bar u_{j,h}) = a(\mathcal P_hu, \bar u_{j,h})
%=a(u, \bar u_{j,h}) = \lambda b(u, \bar u_{j,h}).
%\end{eqnarray*}
%Then the proof is complete.
%\end{proof}
%-------------------------------------------------------------------------------------------------
The following Rayleigh quotient expansion of the eigenvalue error is the tool
to obtain the error estimates of eigenvalue approximations.
\begin{lemma}(\cite{BabuskaOsborn_1989})\label{Rayleigh_quotient_expansion_lem}
Assume $(\lambda,u)$ is an eigenpair of the eigenvalue problem
(\ref{weak_eigenvalue_problem}). Then for any $w \in V\backslash\{0\}$,
the following expansion holds:
\begin{equation}\label{Rayleigh_quotient_expansion}
\frac{a(w,w)}{b(w,w)} - \lambda = \frac{a(w-u,w-u)}{b(w,w)} - \lambda\frac{b(w-u,w-u)}{b(w,w)}.
%\ \ \ \forall u\in M(\lambda).
\end{equation}
\end{lemma}
%-------------------------------------------------------------------------------------------------
The following lemma is similar to the corresponding results in
\cite{BabuskaOsborn_Book,Chatelin,XieZhangOwhadi}.
Since this paper considers the parallel scheme for different eigenpair, we state the error estimates for any single eigenpair.
In order to analyze and understand the proposed numerical algorithms in this paper, we use the following error estimates
from \cite{XieZhangOwhadi} which include only explicit constants. For the proof, please refer to \cite{XieZhangOwhadi}.
%-------------------------------------------------------------------------------------------------
\begin{lemma}(\cite[Lemma 3.3]{XieZhangOwhadi})\label{Error_Estimate_Theorem}
Let  $(\lambda,u)$ denote an exact eigenpair of the eigenvalue problem (\ref{weak_eigenvalue_problem}).
Assume the eigenpair approximation $(\bar\lambda_{i,h},\bar u_{i,h})$ has the property that
$\bar\mu_{i,h}=1/\bar\lambda_{i,h}$ is closest to $\mu=1/\lambda$.
The corresponding spectral projectors $E_{i,h}: V\mapsto {\rm span}\{\bar u_{i,h}\}$
and $E: V\mapsto {\rm span}\{u\}$ are defined as follows
\begin{eqnarray*}
\revise{a(E_{i,h}w,\bar u_{i,h}) = a(w,\bar u_{i,h}),\ \ \ \ {\rm for}\  w\in V,}
\end{eqnarray*}
and
\begin{eqnarray*}
\revise{a(Ew, u) = a(w, u),\ \ \ \ \ {\rm for}\  w\in V.}
\end{eqnarray*}
Then the following error estimate holds
\begin{eqnarray}\label{Energy_Error_Estimate}
\|u-E_{i,h}u\|_a&\leq& \sqrt{1+\frac{\bar\mu_{1,h}}{\delta_{\lambda,h}^2}\eta^2(V_h)}\|(I-\mathcal P_h)u\|_a,
\end{eqnarray}
where $\eta(V_h)$ is define in (\ref{eta_a_h_Def}) and $\delta_{\lambda,h}$ is defined as follows
\begin{eqnarray}\label{Definition_Delta}
\delta_{\lambda,h} &:=& \min_{j\neq i}|\bar\mu_{j,h}-\mu|=\min_{j\neq i} \Big|\frac{1}{\bar\lambda_{j,h}}-\frac{1}{\lambda}\Big|.
\end{eqnarray}
Furthermore, the eigenvector approximation $\bar u_{i,h}$ has following
error estimate in $L^2$-norm
\begin{eqnarray}\label{L2_Error_Estimate}
\|u-E_{i,h}u\|_b &\leq&\Big(1+\frac{\bar\mu_{1,h}}{\delta_{\lambda,h}}\Big)\eta(V_h)\|u-E_{i,h}u\|_a.
\end{eqnarray}
\end{lemma}
\revise{For simplicity of notation, we assume that the eigenvalue gap $\delta_{\lambda,h}$
has a uniform lower bound which is denoted by $\delta_\lambda$ (which can be seen as the
``true" separation of the eigenvalue $\lambda$ from others) in the following parts of this paper.
This assumption is reasonable when the mesh size is small enough. We refer to
\cite[Theorem 4.6]{Saad1} and Lemma \ref{Error_Estimate_Theorem} in this paper for details of the
dependence of error estimates on the eigenvalue gap.
The following alternative error estimates based on Lemma \ref{Error_Estimate_Theorem} are useful for analyzing
the proposed method in Section \ref{Section_3}}.
\begin{lemma}(\cite[Lemma 3.4]{XieZhangOwhadi})\label{Error_Superclose_Theorem}
Under the conditions of Lemma \ref{Error_Estimate_Theorem}, the following error estimates hold
\begin{eqnarray}
%\|u-\bar u_{i,h}\|_a &\leq& \sqrt{2\Big(1+\frac{1}{\lambda_1\delta_\lambda^2}\eta^2(V_h)\Big)}
%\|(I-\mathcal P_h)u\|_a,\label{Err_Norm_1}\\
\|u-\bar u_{i,h}\|_a &\leq& \frac{1}{1-D_\lambda(V_h)\eta(V_h)}\|u-\mathcal P_hu\|_a,\label{Err_Norm_1_Superclose}\\
\|\lambda u - \bar\lambda_{i,h}\bar u_{i,h}\|_b &\leq& C_\lambda(V_h)
\eta(V_h)\|u-\bar u_{i,h}\|_a, \label{Err_Norm_0_Lambda}\\
\|u - \bar u_{i,h}\|_b &\leq& 2\Big(1+\frac{1}{\lambda_1\delta_\lambda}\Big)\eta(V_h)\|u- \bar u_{i,h}\|_a,\label{Error_2}
\end{eqnarray}
where $C_\lambda(V_h)$ and $D_\lambda(V_h)$ are defined as follows
\begin{eqnarray}
C_\lambda(V_h) = 2|\lambda|\Big(1+\frac{1}{\lambda_1\delta_\lambda}\Big)
+ \bar\lambda_{i,h}\sqrt{1+\frac{1}{\lambda_1\delta_\lambda^2}\eta^2(V_h)},
\end{eqnarray}
and
\begin{eqnarray}\label{Definition_D_Lambda}
D_\lambda(V_h) = \frac{1}{\sqrt{\lambda_1}}\left(2|\lambda|\Big(1+\frac{1}{\lambda_1\delta_\lambda}\Big)
+ \bar\lambda_{i,h} \sqrt{1+\frac{1}{\lambda_1\delta_\lambda^2}\eta^2(V_h)}\right).
\end{eqnarray}
\end{lemma}
For the proof, please also refer to \cite{XieZhangOwhadi}.

\section{Parallel augmented subspace method}\label{Section_3}
In this section, we will propose the  parallel augmented subspace method for eigenvalue problems
based on the multilevel correction scheme \cite{LinXie_MultiLevel,Xie_IMA,Xie_JCP,XieZhangOwhadi}.
With the help of the coarse space in multigrid method, the method can transform the solution
of the eigenvalue problem into a series of solutions of the corresponding linear boundary value
problems on the sequence of finite element spaces and eigenvalue problems
on a very low dimensional augmented space.
For different eigenpairs, we can do the correction process independently and it is not
necessary to do orthogonalization in the finest level of fine finite element space. Thus the proposed
algorithm has a good scalability.
Since the eigenvalue problems are only solved in a low dimensional space,
the numerical solution in this new version of augmented subspace method
is not significantly more expensive than
the solution of the corresponding linear boundary value problems.

In order to describe the parallel augmented subspace method clearly, we first introduce
the sequence of finite element
spaces. % and the smoothing properties of appropriate smoothers.
We generate a coarse mesh $\mathcal{T}_H$
with the mesh size $H$ and the coarse linear finite element space $V_H$ is
defined on the mesh $\mathcal{T}_H$. Then we define a sequence of
 triangulations $\mathcal{T}_{h_k}$
of $\Omega\subset \mathcal{R}^d$ as follows.
Suppose that $\mathcal{T}_{h_1}$ (produced from $\mathcal{T}_H$ by some
regular refinements) is given and let $\mathcal{T}_{h_k}$ be obtained
from $\mathcal{T}_{h_{k-1}}$ via one regular refinement step
(produce $\beta^d$ subelements) such that
\begin{eqnarray}\label{mesh_size_recur}
h_k=\frac{1}{\beta}h_{k-1},\ \ \ \ k=2,\cdots,n,
\end{eqnarray}
where the positive number $\beta>1$ denotes the refinement index.
Based on this sequence of meshes, we construct the corresponding
 nested linear finite element spaces such that
\begin{eqnarray}\label{FEM_Space_Series}
V_{H}\subseteq V_{h_1}\subset V_{h_2}\subset\cdots\subset V_{h_n}.
\end{eqnarray}
The sequence of finite element spaces
$V_{h_1}\subset V_{h_2}\subset\cdots\subset V_{h_n}$
 and the finite element space $V_H$ have  the following relations
of approximation accuracy
\begin{eqnarray}\label{delta_recur_relation}
&&\delta(u,V_{h_1}) \leq \sqrt{\lambda}\eta(V_H),\ \ \
\delta(u, V_{h_k})\leq\sqrt{\lambda}\eta(V_{h_k})
\ \ {\rm for} \ k=1,\cdots, n,\label{eta_delta_relation}
\end{eqnarray}
where $u$ is an exact eigenfunction of (\ref{weak_eigenvalue_problem}) corresponding
to the eigenvalue $\lambda$.

\begin{proposition}\label{propCondition_1}
For simplicity of theoretical analysis, we assume the domain $\Omega$ is convex in this paper.
The standard error estimates \cite{BrennerScott,Ciarlet,StrangFix} for the
linear finite element method implies
\begin{eqnarray}
&&\eta(V_{h_k})\leq C h_k,\ \ \delta(u, V_{h_k})\leq C \sqrt{\lambda_i} h_k\ \ \ {\rm for}\ k=1, \cdots, n,\label{eqCondition_1}\\
&&\delta(u,V_{h_k})=\frac{1}{\beta}\delta(u,V_{h_{k-1}}) \ \  \ {\rm for}\ k=2, \cdots, n,\label{delta_recur_relation}
\end{eqnarray}
where $C$ is the constant independent of the mesh size and eigenpair $(\lambda,u)$ of (\ref{weak_eigenvalue_problem}).
\end{proposition}

\subsection{One correction step and efficient implementation}
In order to design the  parallel augmented subspace method, we first introduce
an one correction step in this subsection.

Assume we have obtained an eigenpair approximations
$(\lambda_{h_k}^{(\ell)},u_{h_k}^{(\ell)})\in \mathcal{R}\times V_{h_k}$ for a certain exact eigenpair.
The one correction step is defined by Algorithm \ref{one correction step}
%Now we introduce a type of iteration step as follows to
which can improve the accuracy of the
given eigenpair approximation $(\lambda_{h_k}^{(\ell)},u_{h_k}^{(\ell)})$.

\begin{algorithm}[hbt!]
\label{one correction step}
\caption{One Correction Step}
\begin{enumerate}
\item Define the following linear boundary value problem:
Find $\widehat{u}_{h_k}^{(\ell+1)}\in V_{h_k}$ such that
\begin{equation}\label{correct_source_exact_para}
a(\widehat{u}_{h_k}^{(\ell+1)},v_{h_k}) = \lambda_{h_{k}}^{(\ell)}b(u_{h_{k}}^{(\ell)},v_{h_k}),
\ \  \forall v_{h_k}\in V_{h_k}.
\end{equation}
\revise{Solve (\ref{correct_source_exact_para}) by some multigrid steps
to obtain a new eigenfuction $\widetilde{u}_{h_k}^{(\ell+1)}$. }
%satisfying
%\begin{equation}\label{smooth_correct_process}
%\|\widehat{u}_{h_k}^{(\ell+1)}-\widetilde{u}_{h_k}^{(\ell+1)}\|_a\leq \theta  \|\widehat{u}_{h_k}^{(\ell+1)} -u_{h_{k}}^{(\ell)}\|_a.
%\end{equation}

\item Define a suitable coarse space $V_{H,h_k} = V_H +
{\rm span}\{\widetilde{u}_{h_k}^{(\ell+1)}\}$ and solve the following eigenvalue problem:
Find $(\lambda_{h_k}^{(\ell+1)},u_{h_k}^{(\ell+1)})\in \mathcal{R}\times V_{H,h_k}$
such that $a(u_{h_k}^{(\ell+1)},u_{h_k}^{(\ell+1)})=1$ and
\begin{equation}\label{parallel_correct_eig_exact}
a(u_{h_k}^{(\ell+1)},v_{H,h_k}) = \lambda_{h_k}^{(\ell+1)}b(u_{h_k}^{(\ell+1)},v_{H,h_k}),
\ \ \ \ \ \forall v_{H,h_k}\in V_{H,h_k}.
\end{equation}
\revise{Solve (\ref{parallel_correct_eig_exact}) and the output $(\lambda_{h_k}^{(\ell+1)},u_{h_k}^{(\ell+1)})$
is chosen such that $u_{h_k}^{(\ell+1)}$ has the largest component in ${\rm span}\{\ \widetilde{u}_{h_k}^{(\ell+1)}\}$ among
all eigenfunctions of (\ref{parallel_correct_eig_exact}).}
\end{enumerate}
Summarize the above two steps by defining
\begin{eqnarray*}
(\lambda_{h_k}^{(\ell+1)},u_{h_k}^{(\ell+1)}) =
{\tt Correction}(V_H,V_{h_k},\lambda_{h_{k}}^{(\ell)},u_{h_{k}}^{(\ell)}).
 \end{eqnarray*}
\end{algorithm}

%--------------------------------------------------------------------------------------------
%For simplicity of notation, we assume that the eigenvalue gap $\delta_{\lambda,h}$
%has a uniform lower bound which is denoted by $\delta_\lambda$ (which can be seen as the
%``true" separation of the eigenvalue $\lambda$ from others) in the following parts of this paper.
%This assumption is reasonable when the mesh size $H$ is small enough. We refer to
%\cite[Theorem 4.6]{Saad1} and Theorem \ref{Error_Estimate_Theorem} in this paper for details on the
%dependence of error estimates on the eigenvalue gap.

\comm{In this section,  we assume the concerned eigenpair
approximation $(\lambda_{h_k}^{(\ell)}, u_{h_k}^{(\ell)})$  with different superscript is closet \revise{to an exact eigenpair} $(\bar\lambda_{h_k}, \bar u_{h_k})$
of (\ref{Weak_Eigenvalue_Discrete}) and $(\lambda, u)$ of (\ref{weak_eigenvalue_problem}) in this section.}
\begin{theorem}\label{Error_Estimate_One_Smoothing_Theorem}
Assume there exists an exact eigenpair $(\bar\lambda_{h_k}, \bar u_{h_k})$ such
that the  eigenpair approximation $(\lambda_{h_k}^{(\ell)},u_{h_k}^{(\ell)})$ satisfies
$\|u_{h_k}^{(\ell)}\|_a=1$ and
\begin{eqnarray}\label{Estimate_h_k_b}
\|\bar\lambda_{h_k}\bar u_{h_k}-\lambda_{h_k}^{(\ell)}u_{h_k}^{(\ell)}\|_b \leq C_1 \eta (V_H)\|\bar u_{h_k}-u_{h_k}^{(\ell)}\|_a,
\end{eqnarray}
for some constant $C_1$. The multigrid iteration for the linear equation (\ref{correct_source_exact_para})
has the following uniform contraction rate
\begin{eqnarray}\label{Contraction_Rate}
\revise{\|\widehat u_{h_k}^{(\ell+1)}-\widetilde u_{h_k}^{(\ell+1)}\|_a\leq \theta \|u_{h_k}^{(\ell)}-\widehat u_{h_k}^{(\ell+1)}\|_a},
\end{eqnarray}
with $\theta<1$ independent of $k$ and $\ell$.

Then the eigenpair approximation
$(\lambda_{h_k}^{(\ell+1)},u_{h_k}^{(\ell+1)})\in\mathcal R\times V_{h_k}$ produced by
Algorithm \ref{one correction step} satisfies
\begin{eqnarray}
\|\bar u_{h_k}-u_{h_k}^{(\ell+1)}\|_a &\leq & \gamma \|\bar u_{h_k}-u_{h_k}^{(\ell)}\|_a,\label{Estimate_h_k_1_a}\\
\|\bar\lambda_{h_k}\bar u_{h_k}-\lambda_{h_k}^{(\ell+1)}u_{h_k}^{(\ell+1)}\|_b&\leq&
\bar C_\lambda \eta (V_H)\|\bar u_{h_k}-u_{h_k}^{(\ell+1)}\|_a,\label{Estimate_h_k_1_b}
\end{eqnarray}
where the constants $\gamma$, $\bar C_\lambda$ and $\bar D_\lambda$ are defined as follows
\begin{eqnarray}
\gamma &=& \frac{1}{1-\bar D_\lambda \eta(V_H)}
\Big(\theta+(1+\theta)\frac{C_1}{\sqrt{\lambda_1}}\eta (V_H)\Big)\,,\label{Gamma_Definition}\\
\bar C_\lambda &=& 2|\lambda|\Big(1+\frac{1}{\lambda_1\delta_\lambda}\Big)
+ \bar\lambda_{i,H}\sqrt{1+\frac{1}{\lambda_1\delta_\lambda^2}\eta^2(V_H)},\label{Definition_C_Bar}\\
\bar D_\lambda &=& \frac{1}{\sqrt{\lambda_1}}\left(2|\lambda|\Big(1+\frac{1}{\lambda_1\delta_\lambda}\Big)
+ \bar\lambda_{i,H}\sqrt{1+\frac{1}{\lambda_1\delta_\lambda^2}\eta^2(V_H)}\right).\label{Definition_D_Lambda_Bar}
\end{eqnarray}
\end{theorem}
%------------------------------------------------------------------------------------------------
\begin{proof}
From \eqref{Smallest_Eigenvalue}, (\ref{Weak_Eigenvalue_Discrete}),
(\ref{correct_source_exact_para}) and (\ref{Contraction_Rate}),
we have for  $w\in V_{h_k}$
\begin{eqnarray}\label{One_Correction_1}
&&a(\bar u_{h_k}-\widehat{u}_{h_k}^{(\ell+1)}, w)
=b\big((\bar\lambda_{h_k}\bar u_{h_k}-\lambda_{h_k}^{(\ell)}u_{h_k}^{(\ell)}),w\big)\nonumber\\
&&\leq\|\bar\lambda_{h_k}\bar u_{h_k}-\lambda_{h_k}^{(\ell)}u_{h_k}^{(\ell)}\|_b\|w\|_b
\leq C_1\eta(V_H)\|\bar u_{h_k}-u_{h_k}^{(\ell)}\|_a \|w\|_b \nonumber\\
&&\leq \frac{1}{\sqrt{\lambda_1}}C_1\eta(V_H)\|\bar u_{h_k}-u_{h_k}^{(\ell)}\|_a \|w\|_a.
\end{eqnarray}
Taking $w = \bar u_{h_k}-\widehat{u}_{h_k}^{(\ell+1)}$ in (\ref{One_Correction_1})
leads to the following inequality
%we deduce from  \eqref{Estimate_h_k_b} that
\begin{eqnarray}\label{One_Correction_2}
\|\bar u_{h_k}-\widehat{u}_{h_k}^{(\ell+1)}\|_a &\leq&\frac{C_1}{\sqrt{\lambda_1}}\eta (V_H)
\| \bar u_{h_k}-u_{h_k}^{(\ell)}\|_a.
\end{eqnarray}
Using \eqref{Contraction_Rate}, \eqref{One_Correction_2} and triangle inequality, we deduce
following estimates for $\|\bar u_{h_k}-\widetilde{u}_{h_k}^{(\ell+1)}\|_a$
\begin{eqnarray}\label{Error_1}
\|\bar u_{h_k}-\widetilde{u}_{h_k}^{(\ell+1)}\|_a&\leq& \|\bar u_{h_k}-\widehat{u}_{h_k}^{(\ell+1)}\|_a
+ \|\widetilde{u}_{h_k}^{(\ell+1)}-\widehat{u}_{h_k}^{(\ell+1)}\|_a\nonumber\\
&\leq& \|\bar u_{h_k}-\widehat{u}_{h_k}^{(\ell+1)}\|_a  + \theta \|\widehat{u}_{h_k}^{(\ell+1)}-u_{h_k}^{(\ell)}\|_a\nonumber\\
&\leq& \|\bar u_{h_k}-\widehat{u}_{h_k}^{(\ell+1)}\|_a +
\theta \|\widehat{u}_{h_k}^{(\ell+1)}- \bar u_{h_k}\|_a
+\theta\| \bar {u}_{h_k}-u_{h_k}^{(\ell)}\|_a\nonumber\\
&\leq& (1+\theta)\|\bar u_{h_k}-\widehat{u}_{h_k}^{(\ell+1)}\|_a
+\theta\| \bar u_{h_k}-u_{h_k}^{(\ell)}\|_a\nonumber\\
&\leq& \Big(\theta+(1+\theta)\frac{C_1}{\sqrt{\lambda_1}}\eta (V_H)\Big)\|\bar u_{h_k}-u_{h_k}^{(\ell)}\|_a.
\end{eqnarray}
\revise{Since $V_H\subset V_{H,h_k}$, the following inequality holds
\begin{eqnarray}\label{Inequality_11}
\eta(V_{H,h_k}) \leq \eta(V_H),\ \ C_\lambda(V_{H,h_k})\leq C_\lambda(V_H)=:\bar C_\lambda,\ \
D_\lambda(V_{H,h_k})\leq D_\lambda(V_H)=:\bar D_\lambda.
\end{eqnarray}}
The eigenvalue problem \eqref{parallel_correct_eig_exact} can be seen as
a low dimensional subspace approximation of the eigenvalue problem (\ref{Weak_Eigenvalue_Discrete}).
From  (\ref{Err_Norm_1_Superclose}), Lemmas \ref{Error_Estimate_Theorem} and \ref{Error_Superclose_Theorem},  (\ref{Error_1}), (\ref{Inequality_11}),  the following error
estimates hold % we obtain that
\begin{eqnarray*}\label{Error_u_u_h_2}
&&\|\bar u_{h_k}-u_{h_k}^{(\ell+1)}\|_a \leq \frac{1}{1-D_\lambda(V_{H,h_k}) \eta(V_{H,h_k})}\inf_{v_{H,h_k}\in V_{H,h_k}}\|\bar u_{h_k}-v_{H,h_k}\|_a\nonumber\\
&&\leq \frac{1}{1-\bar D_\lambda \eta(V_H)}\|\bar u_{h_k}-\widetilde{u}_{h_k}^{(\ell+1)}\|_a
\leq \gamma \|\bar u_{h_k}-u_{h_k}^{(\ell)}\|_a,
\end{eqnarray*}
and
\begin{eqnarray*}\label{Error_u_u_h_2_Negative}
\|\bar{\lambda}_{h_k}\bar u_{h_k}-\lambda_{h_k}^{(\ell+1)} u_{h_k}^{(\ell+1)}\|_b
\leq C_\lambda(V_{H,h_k})\eta(V_{H,h_k})\|\bar u_{h_k}-u_{h_k}^{(\ell+1)}\|_a
\leq \bar C_\lambda\eta(V_H)\|\bar u_{h_k}-u_{h_k}^{(\ell+1)}\|_a.
\end{eqnarray*}
Then we have the desired results (\ref{Estimate_h_k_1_a}) and (\ref{Estimate_h_k_1_b})
and conclude the proof.
\end{proof}
\begin{remark}\label{Remark_Gamma}
\revise{Since the multigrid iteration step has uniform convergence rate (independent of the mesh size), there exist
$\theta<1$ such that (\ref{Contraction_Rate}) holds. }  Definition (\ref{Gamma_Definition}),  Lemmas \ref{Error_Estimate_Theorem}
and \ref{Error_Superclose_Theorem} imply that $\gamma$ is less than $1$ when $\eta(V_H)$ is small enough.
If $\lambda$ is large or the spectral gap $\delta_\lambda$ is small,
then we need to use a smaller $\eta(V_H)$ or $H$. Furthermore, we can
increase the multigrid steps to reduce $\theta$ and then $\gamma$.
These theoretical restrictions do not limit practical applications where (in numerical implementations),
$H$ is simply chosen (just) small enough so that the  number of elements of
corresponding coarsest space (just) exceeds  the required number of eigenpairs
($H$ and the coarsest space are adapted to the  number of eigenpairs to be computed).
\end{remark}

We would like to point out that the given eigenpair
$(\lambda_{h_k}^{(\ell)},u_{h_k}^{(\ell)})$ is not necessary to be the one corresponding to the smallest eigenvalue.
So when we need to solve more than one eigenpairs, the one correction step defined by
Algorithm \ref{one correction step}  can be carried out independently
for every eigenpair and there exists no
data exchanging. This property means that we can avoid doing the time-consuming orthogonalization
in the high dimensional space $V_{h_k}$.

Now, let us give details for the second step of Algorithm \ref{one correction step}.
Solving the eigenvalue problem (\ref{parallel_correct_eig_exact}) provides several
eigepairs. Since the desired eigenvalue maybe not the first (smallest) one, we should
choose the suitable or the desired eigenpair from the ones of (\ref{parallel_correct_eig_exact}).
Let us consider the details to choose the desired eigenpair which has the best accuracy among all
the eigenpairs of eigenvalue problem (\ref{parallel_correct_eig_exact}).
For this aim, we come to consider the matrix version of the small scaled eigenvalue problem (\ref{parallel_correct_eig_exact}).
Let  $N_H$ and $\{\phi_{j,H}\}_{1\leq j\leq N_H}$ denote the dimension and Lagrange basis functions for the coarse finite element space $V_H$.
The function in $V_{H,h_k}$ can be denoted by $u_{H,h_k}=u_H+\alpha_k \widetilde u_{h_k}$.
Solving eigenvalue problem (\ref{parallel_correct_eig_exact}) is to obtain the function $u_H\in V_H$ and the value $\alpha_k\in \mathcal R$.
Let $u_H=\sum_{j=1}^{N_H}u_j\phi_{j,H}$ and define the vector $\mathbf u_H$ as $\mathbf u_H=[u_1,\cdots, u_{N_H}]^T$. Based on the structure of the space $V_{H,h_k}$, the matrix version of
the eigenvalue problem (\ref{parallel_correct_eig_exact}) can be written as follows
\begin{equation}\label{Eigenvalue_H_h}
\left(
\begin{array}{cc}
A_H & b_{H,h_k}\\
b_{H,h_k}^T&\beta_k
\end{array}
\right)
\left(
\begin{array}{c}
\mathbf u_H\\
\alpha_k
\end{array}
\right)
=\lambda_{h_k}
\left(
\begin{array}{cc}
M_H& c_{H,h_k}\\
c_{H,h_k}^T & \zeta_k
\end{array}
\right)
\left(
\begin{array}{c}
\mathbf u_H\\
\alpha_k
\end{array}
\right),
\end{equation}
where $\mathbf u_H\in \mathcal R^{N_H}$, $\alpha_k\in\mathcal R$, column
vectors $b_{H,h_k}\in\mathcal R^{N_H}$ and $c_{H,h_k}\in\mathcal R^{N_H}$,
scalars $\beta_k$ and $\zeta_k$ are defined as follows
\begin{eqnarray*}
&&b_{H,h_k} = [a(\phi_{j,H}, \widetilde u_{h_k})]_{1\leq j\leq N_H}\in \mathcal R^{N_H}, \ \ \
c_{H,h_k}   = [b(\phi_{j,H}, \widetilde u_{h_k})]_{1\leq j\leq N_H}\in \mathcal R^{N_H},\\
&&\beta_k   =  a(\widetilde u_{h_k},\widetilde u_{h_k})\in\mathcal R,\ \ \ \ \
\zeta_k     =  b(\widetilde u_{h_k},\widetilde u_{h_k})\in \mathcal R.
\end{eqnarray*}

In the practical calculation, the desired eigenpair $(\lambda_{h_k}^{(\ell+1)},u_{h_k}^{(\ell+1)})$
may be not the eigenpair corresponding to the smallest eigenvalue
and solving eigenvalue problem (\ref{parallel_correct_eig_exact}) will produce a series of $[\mathbf u_H; \alpha_k]^T$.
In this case, we need to choose the approximate solution which has
the largest component in the direction $\text{span}\{ \widetilde u_{h_k}\}$
which is the desired eigenpair in the one correction step defined by Algorithm \ref{one correction step}.
Since there holds
\begin{eqnarray}
|b(u_H+\alpha_k \widetilde u_{h_k}, \widetilde u_{h_k})| &=& |b(u_H, \widetilde u_{h_k})
+\alpha_k b(\widetilde u_{h_k}, \widetilde u_{h_k})|\nonumber\\
&=& |\mathbf u_H\cdot c_{H,h_k} + \alpha_k\zeta_k|,
\end{eqnarray}
%After obtaining the approximate solutions $(\mathbf u_H, \alpha_k)$,
we only need to
calculate $\mathbf u_H\cdot c_{H,h_k} + \alpha_k\zeta_k$ for every eigenvector
$[\mathbf u_H; \alpha_k]^T$ which are obtained
by solving (\ref{parallel_correct_eig_exact}) numerically, and then choose the one
with the largest absolute value as the desired solution.

In the second step of Algorithm \ref{one correction step}, we can use the shift-inverse technique since
an approximate eigenpair has been obtained in the previous step. Furthermore, we can use the different
level of space to act as the coarse space $V_H$  in the one correction step for different eigenvalue.

%------------------------------------------------------------------------------------------------
\subsection{Parallel augmented subspace method}
In this subsection, we introduce a  parallel augmented subspace method
based on the one correction step defined in Algorithm \ref{one correction step}.

Here, the aim of the parallel method is to compute $m$ eigenpair approximations of (\ref{weak_eigenvalue_problem}).
For simplicity, we denote
the desired eigenpairs by $(\lambda_1,u_1), \cdots, (\lambda_m,u_m)$
and assume there exist $m$ processes denoted by $\{P_1, \cdots, P_m\}$ for the parallel computing.
When the number of processes is not equal to the number of desired eigenparis,
in order to improve the parallel efficiency, the distribution of desired eigenparis onto
the processes should be equal as far as possible to
arrive the load balancing. About this point, we refer to the concerned papers for
load balancing. % \cite{loading}.
The corresponding parallel augmented subspace algorithm is described in Algorithm \ref{Para_Multigrid}. From Algorithm \ref{Para_Multigrid},
the computation for $m$ eigenpairs is decomposed into $m$ processes.
\begin{algorithm}[ht]\label{Para_Multigrid}
\caption{Parallel Augmented Subspace Scheme}
\begin{enumerate}
\item Solve the following eigenvalue problem:
Find $(\lambda_{h_1}, u_{h_1})\in \mathcal{R}\times V_{h_1}$ such that $a(u_{h_1},u_{h_1})=1$ and
\begin{equation}\label{step1}
a(u_{h_1}, v_{h_1}) = \lambda_{h_1} b(u_{h_1}, v_{h_1}), \quad \forall v_{h_1}\in  V_{h_1}.
\end{equation}
Solve eigenvalue problem (\ref{step1}) on the first process to get initial eigenpair approximations
$(\lambda_{i,h_1},u_{i,h_1})$ $\in\mathcal{R}\times V_{h_1}$, $\ i=1,\cdots,m$,
which are approximations for the desired eigenpairs $(\lambda_i,u_i)$, $i=1$, $\cdots$, $m$.
Then the  eigenpair approximations  $(\lambda_{i,h_1},u_{i,h_1})$, $\ i=2,\cdots,m$ are delivered to other $m-1$ processes.
\item For $i=1,\cdots,m$, do the following multilevel correction steps on the process $P_i$ in the parallel way

\begin{itemize}
\item [(A).] For $k= 1, \cdots, n-2$, do the following iteration:
\begin{itemize}
\item [(a).]\ Set $(\lambda_{i,h_{k+1}}^{(0)}, u_{i,h_{k+1}}^{(0)}):=(\lambda_{i,h_k},u_{i,h_k})$.
\item [(b).]\ For $\ell = 0, \cdots, \revise{\varpi-1}$, do the following one correction steps
\begin{eqnarray*}
(\lambda_{i,h_{k+1}}^{(\ell+1)},u_{i,h_{k+1}}^{(\ell+1)}) =
{\tt Correction}(V_H,V_{h_{k+1}},\lambda_{i,h_{k+1}}^{(\ell)},u_{i,h_{k+1}}^{(\ell)}).
\end{eqnarray*}
\item [(c).]\ Set $(\lambda_{i,h_{k+1}},u_{i,h_{k+1}}):= (\lambda_{i,h_{k+1}}^{(\varpi)},u_{i,h_{k+1}}^{(\varpi)})$
as the output in the $k+1$-th level space $V_{h_{k+1}}$.
\end{itemize}
\item [(B).] Do the following iterations on the finest level space $V_{h_n}$:
\begin{itemize}
\item [(a).]\ Set $(\lambda_{i,h_n}^{(0)}, u_{i,h_n}^{(0)}):=(\lambda_{i,h_{n-1}},u_{i,h_{n-1}})$.
\item [(b).]\ For $\ell = 0, \cdots, \revise{\varpi_n-1}$, do the following one correction steps
\begin{eqnarray*}
(\lambda_{i,h_n}^{(\ell+1)},u_{i,h_n}^{(\ell+1)}) =
{\tt Correction}(V_H,V_{h_n},\lambda_{i,h_n}^{(\ell)},u_{i,h_n}^{(\ell)}).
\end{eqnarray*}
\item [(c).]\ Set $(\lambda_{i,h_n},u_{i,h_n}):= (\lambda_{i,h_n}^{(\varpi_n)},u_{i,h_n}^{(\varpi_n)})$
  as the output in the $n$-th level space $V_{h_n}$.
\end{itemize}
\end{itemize}
\end{enumerate}
Finally, we obtain eigenpair approximations
$\{(\lambda_{i,h_n},u_{i,h_n})\}_{i=1}^{m}\in \mathcal{R}\times V_{h_n}$.
\end{algorithm}
%--------------------------------------------------------------------------------------------------------------------
%\begin{remark}

In order to make the initial eigenfunction approximations $u_{1,h_1}, \cdots, u_{m,h_1}$ be
orthogonal each other, in the first step of Algorithm \ref{Para_Multigrid},
the eigenvalue problem is solved in the first process.
We adopt this strategy since (\ref{step1}) is a low dimensional eigenvalue problem
compared with the one in the finest  space.
Similarly to the idea in the full mulgrid method for boundary value problems,
the step 2. (A)  in Algorithm \ref{Para_Multigrid}
is used to give an initial eigenpair approximation in the finest space $V_{h_n}$.

Algorithm \ref{Para_Multigrid} shows the idea to design the parallel method for different eigenpairs.
In each process, the main computation in the one correction step defined
by Algorithm \ref{one correction step} is to solve the linear
equation (\ref{correct_source_exact_para}) in the fine space $V_{h_k}$.
It is an easy and direct idea to use the parallel scheme to solve this linear
equation based on the mesh distribution on different processes.
This type of parallel method is well-developed and there exist many mature software packages such as
Parallel Hierarchy Grid (PHG). But we would like to say this is another sense of parallel scheme
and this paper is concerned with the  parallel method for different eigenpair. These discussion means
we can design a two level parallel scheme for the eigenvalue problem solving.

%In Algorithm \ref{Para_Multigrid}, If $p\geq m$, we can solve several eigenvalues at one precess. Otherwise,
%we have at least one process to do the correction step for every eigenpair
%on each level. Thus we can solve the equations in correction step
%by parallel computing based on mesh distribution.  So we actually derive a two level parallel algorithm.

%\end{remark}

%---------------------------------------------------------------------------------------------
\begin{theorem}\label{Error_Full_Multigrid_Theorem}
%Assume there holds the following condition
\revise{Assume the numer $\varpi$ of the one correction steps satisfies}
\begin{eqnarray}\label{Convergence_Condition}
\gamma^\varpi\beta <1.
\end{eqnarray}
After implementing Algorithm \ref{Para_Multigrid}, the resulting
eigenpair approximation $(\lambda_{i,h_n},u_{i,h_n})$ has following error estimates
\begin{eqnarray}
\|\bar u_{i,h_n}-u_{i,h_n}\|_a &\leq&
\frac{2\beta}{1-\bar D_\lambda\eta(V_H)}\gamma^{\varpi_n} \Big(1 + \frac{\gamma^{\varpi}\beta}{1-\gamma^{\varpi}\beta}\Big)\delta(u,V_{h_n}),\label{FM_Err_fun_H1}\\
\|\bar u_{i,h_n}- u_{i,h_n}\|_b
&\leq&2\Big(1+\frac{1}{\lambda_1\delta_\lambda}\Big)\eta(V_H)
\|\bar u_{i,h_n} -  u_{i,h_n}\|_a,\label{FM_Err_fun_L2}\\
|\bar\lambda_{i,h_n}-\lambda_{i,h_n}| &\leq&
 \lambda_{i,h_n}\|\bar u_{i,h_n} - u_{i,h_n}\|_a^2.\label{FM_Err_Eigen}
\end{eqnarray}
\end{theorem}
%---------------------------------------------------------------------------------------------
%--------------------------------------------------------------------------
\begin{proof}
Define $e_{i,k}:=\bar u_{i,h_k}-u_{i,h_k}$. From step 1
in Algorithm \ref{Para_Multigrid}, it is obvious $e_{i,1}=0$.
Then the assumption (\ref{Estimate_h_k_b}) in
Theorem \ref{Error_Estimate_One_Smoothing_Theorem} is satisfied for $k=1$. From the definitions of
Algorithms \ref{one correction step} and \ref{Para_Multigrid},
Theorem \ref{Error_Estimate_One_Smoothing_Theorem} and recursive argument,
the assumption (\ref{Estimate_h_k_b}) holds for each level of space $V_{h_k}$ ($k=1, \cdots, n$)
with $C_1 = \bar C_\lambda$ in (\ref{Definition_C_Bar}).
Then the convergence rate (\ref{Estimate_h_k_1_a}) is valid for all $k=1, \cdots, n$
and $\ell = 0, \cdots, \varpi-1$.

For $k=2,\cdots,n-1$, by (\ref{Delta_V_h_P_h}), (\ref{Err_Norm_1_Superclose}),
Theorem \ref{Error_Estimate_One_Smoothing_Theorem}, (\ref{Inequality_11})
and recursive argument, we have
\begin{eqnarray}\label{FM_Estimate_1}
\|e_{i,k}\|_a&\leq& \gamma^\varpi\|\bar u_{i,h_k}- u_{i,h_{k-1}}\|_a
\leq \gamma^\varpi\big(\|\bar u_{i,h_k}-\bar u_{i,h_{k-1}}\|_a
+\|\bar u_{i,h_{k-1}}-u_{i,h_{k-1}}\|_a\big)\nonumber\\
&\leq& \gamma^\varpi\big(\|\bar u_{i,h_k}-u\|_a+\|u-\bar u_{i,h_{k-1}}\|_a
+\|\bar u_{i,h_{k-1}}-u_{i,h_{k-1}}\|_a\big)\nonumber\\
&=& \gamma^\varpi \left(\frac{1}{1-\bar D_\lambda\eta(V_H)}\big(\delta(u, V_{h_k})+\delta(u, V_{h_{k-1}})\big)+\|e_{i,k-1}\|_a\right)\nonumber\\
&\leq& \gamma^\varpi \left(\frac{2}{1-\bar D_\lambda\eta(V_H)}\delta(u,V_{h_{k-1}})+\|e_{i,k-1}\|_a\right).
\end{eqnarray}

By Proposition \ref{propCondition_1}, (\ref{Convergence_Condition}) and  iterating inequality (\ref{FM_Estimate_1}),
the following inequalities hold
\begin{eqnarray}\label{Initial_Error}
\|e_{i,n-1}\|_a &\leq& \frac{2}{1-\bar D_\lambda\eta(V_H)}\left(\gamma^\varpi\delta(u,V_{h_{n-2}})+
\cdots +\gamma^{(n-2)\varpi}\delta(u,V_{h_1})\right)\nonumber\\
&\leq& \frac{2}{1-\bar D_\lambda\eta(V_H)}\sum_{k=1}^{n-2} \gamma^{(n-1-k)\varpi}\delta(u,V_{h_k})\nonumber\\
&\leq& \frac{2}{1-\bar D_\lambda\eta(V_H)}\sum_{k=1}^{n-2} \big(\gamma^{\varpi}\beta\big)^{n-1-k}\delta(u,V_{h_{n-1}})\nonumber\\
&\leq& \frac{2}{1-\bar D_\lambda\eta(V_H)}\frac{\gamma^{\varpi}\beta}{1-\gamma^{\varpi}\beta}\delta(u,V_{h_{n-1}}).
\end{eqnarray}
Then the combination of  Theorem \ref{Error_Estimate_One_Smoothing_Theorem}, (\ref{Initial_Error})
and Algorithm \ref{Para_Multigrid} leads to the following error estimates
\begin{eqnarray*}\label{Final_Error}
\|e_{i,n}\|_a &\leq& \gamma^{\varpi_n} \|\bar u_{i,h_n}- u_{i,h_{n-1}}\|_a
\leq \gamma^{\varpi_n} (\|\bar u_{i,h_n}- \bar u_{i,h_{n-1}}\|_a+ \|\bar u_{i,h_{n-1}}
-u_{i,h_{n-1}}\|_a)\nonumber\\
&\leq& \gamma^{\varpi_n} \left(\frac{2}{1-\bar D_\lambda\eta(V_H)}\delta(u,V_{h_{n-1}})+ \|e_{i,n-1}\|_a\right)\nonumber\\
&\leq& \gamma^{\varpi_n} \left(\frac{2}{1-\bar D_\lambda\eta(V_H)}\delta(u,V_{h_{n-1}}) + \frac{2}{1-\bar D_\lambda\eta(V_H)}\frac{\gamma^{\varpi}\beta}{1-\gamma^{\varpi}\beta}\delta(u,V_{h_{n-1}})\right)\nonumber\\
&\leq & \frac{2}{1-\bar D_\lambda\eta(V_H)}\gamma^{\varpi_n} \Big(1 + \frac{\gamma^{\varpi}\beta}{1-\gamma^{\varpi}\beta}\Big)\delta(u,V_{h_{n-1}})\nonumber\\
&\leq & \frac{2\beta}{1-\bar D_\lambda\eta(V_H)}\gamma^{\varpi_n} \Big(1 + \frac{\gamma^{\varpi}\beta}{1-\gamma^{\varpi}\beta}\Big)\delta(u,V_{h_n}).
\end{eqnarray*}
%For such choices of $\varpi$ and $\varpi_n$,
This means we arrive at the desired result (\ref{FM_Err_fun_H1}).

From \eqref{Rayleigh_quotient_expansion}, (\ref{Error_2}) and  \eqref{FM_Err_fun_H1}, we have following error estimates
\begin{eqnarray*}
\|\bar u_{i,h_n} - u_{i,h_n}\|_b &\leq& 2\Big(1+\frac{1}{\lambda_1\delta_\lambda}\Big)\eta(V_H)
\|\bar u_{i,h_n} -  u_{i,h_n}\|_a,\nonumber\\
|\bar\lambda_{i,h_n}-\lambda_{i,h_n}| &\leq& \frac{\|\bar u_{i,h_n} - u_{i,h_n}\|_a^2}{\|u_{i,h_n}\|_b^2}
\leq \lambda_{i,h_n}\|\bar u_{i,h_n} - u_{i,h_n}\|_a^2,
\end{eqnarray*}
which are the desired results (\ref{FM_Err_fun_L2}) and (\ref{FM_Err_Eigen}).
\end{proof}
%-------------------------------------------------------------------------------------------------
\begin{remark}\label{Remark_Orthogonal}
The proof of Theorem \ref{Error_Full_Multigrid_Theorem} implies that the assumption (\ref{Estimate_h_k_b})
in Theorem \ref{Error_Estimate_One_Smoothing_Theorem}
holds for $C_1=\bar C_\lambda$ in each level of space $V_{h_k}$ ($k=1, \cdots, n$).
The structure of Algorithm \ref{Para_Multigrid}, shows that $\bar C_\lambda$ does not change
as the algorithm progresses from the initial space $V_{h_1}$ to the finest one $V_{h_n}$.
%-------------------------------------------------------------------------------------------------

From the estimate (\ref{FM_Err_fun_H1}), it can be observed that the final algebraic accuracy depends strongly on
$\gamma^{\varpi_n}$. Furthermore,  increasing
$\varpi$ on the coarse levels spaces $V_{h_2},\cdots, V_{h_{n-1}}$ can not improve the final algebraic accuracy.
For this reason, we always set $\varpi=1$ on the coarse level spaces $V_{h_2},\cdots, V_{h_{n-1}}$.

Now we briefly analyze the orthogonality of different eigenfunctions obtained by Algorithm \ref{Para_Multigrid}.
Suppose $u_{i,h_n} =  \bar u_{i,h_n}+r_i$ and $u_{j,h_n} = \bar u_{j,h_n}+r_j$ corresponding
to $\bar\lambda_{i,h_n}\neq \bar\lambda_{j,h_n}$.
By Theorem \ref{Error_Full_Multigrid_Theorem}, we have the error estimates for $r_i$ and $r_j$.
Furthermore, the orthogonality of $u_{i,h_n}$ and $u_{j,h_n}$ has following estimate
\begin{eqnarray*}
b(u_{i,h_n}, u_{j,h_n}) &=& b(\bar u_{i,h_n}+r_i, \bar u_{j,h_n}+r_j)
=b(r_i, \bar u_{j,h_n})+b(\bar u_{i,h_n}, r_j)+b(r_i, r_j)\nonumber\\
&\leq& \|r_i\|_b+\|r_j\|_b+\|r_i\|_b\| r_j\|_b.
\end{eqnarray*}
So Algorithm \ref{Para_Multigrid} can keep the orthogonality for different eigenfunction
when we do enough
correction steps ($\varpi_n$ is enough large) such that the algebraic accuracy
is enough small in the finest space $V_{h_n}$.
\end{remark}
%------------------------------------------------------------------------------------------------------
\begin{theorem}\label{Error_Estimate_Theorem_Final}
Under the conditions of Theorem \ref{Error_Full_Multigrid_Theorem},
after implementing Algorithm \ref{Para_Multigrid}, there exists an eigenpair $(\lambda, u)$ of (\ref{weak_eigenvalue_problem}) such that
the eigenpair approximation $(\lambda_{i,h_n},u_{i,h_n})$ has  following
error estimates for $i=1,\cdots,m$
\begin{eqnarray}
&&\|u-u_{i,h_{n}}\|_a \leq \frac{1}{1-\bar D_\lambda\eta(V_H)} \left(1+2\beta\gamma^{\varpi_n} \Big(1 + \frac{\gamma^{\varpi}\beta}{1-\gamma^{\varpi}\beta}\Big)\right)\delta(u,V_{h_n}),\label{Final_Error_fun_H1}\\
&&\|u-u_{i,h_{n}}\|_b\leq 2\Big(1+\frac{1}{\lambda_1\delta_\lambda}\Big)\eta(V_H)\frac{1}{1-\bar D_\lambda\eta(V_H)} \Big(1+\Big.\nonumber\\
&&\quad\quad\quad\quad\quad\quad\quad\quad \quad\quad\quad\quad\quad\quad\quad\quad\ \ \Big.2\beta \gamma^{\varpi_n}
\Big(1 + \frac{\gamma^{\varpi}\beta}{1-\gamma^{\varpi}\beta}\Big) \Big)\delta(u,V_{h_n}),
\label{Final_Error_fun_L2}\\
&&|\lambda -\lambda_{i,h_{n}}| \leq \lambda_{i,h_n}\left(\frac{1}{1-\bar D_\lambda\eta(V_H)}\right)^2 \left(1+2\beta\gamma^{\varpi_n} \Big(1 + \frac{\gamma^{\varpi}\beta}{1-\gamma^{\varpi}\beta}\Big)\right)^2\delta^2(u,V_{h_n}).\ \ \label{Final_Error_eigen}
\end{eqnarray}
\end{theorem}
%--------------------------------------------------------------------------
\begin{proof}
From (\ref{Delta_V_h_P_h}), (\ref{Err_Norm_1_Superclose}), (\ref{Inequality_11}), Theorem \ref{Error_Full_Multigrid_Theorem}
 and (\ref{Convergence_Condition}), we have following estimates
\begin{eqnarray*}\label{Estimate_Final}
&&\|u-u_{i,h_n}\|_a  \leq \| u - \bar u_{i,h_n}\|_a + \|\bar u_{i,h_n} - u_{i,h_n}\|_a\nonumber\\
%&\leq& \sqrt{2\Big(1+\frac{1}{\lambda_1\delta_\lambda^2}\eta^2(V_{h_n})\Big)}
%\|(I-\mathcal P_{h_n})u_i\|_a + 2\beta \gamma^{\varpi_n} \Big(1 + \frac{2\gamma^{\varpi}\beta}{1-\gamma^{\varpi}\beta}\Big)\delta(u,V_{h_n})\nonumber\\
&\leq& \frac{1}{1-\bar D_\lambda\eta(V_H)}
\|(I-\mathcal P_{h_n})u\|_a + \frac{2\beta}{1-\bar D_\lambda\eta(V_H)}\gamma^{\varpi_n}
\Big(1 + \frac{\gamma^{\varpi}\beta}{1-\gamma^{\varpi}\beta}\Big)\delta(u,V_{h_n})\nonumber\\
&=&\frac{1}{1-\bar D_\lambda\eta(V_H)} \left(1+2\beta\gamma^{\varpi_n} \Big(1 + \frac{\gamma^{\varpi}\beta}{1-\gamma^{\varpi}\beta}\Big)\right)\delta(u,V_{h_n}).
%&=& \left(\sqrt{2\Big(1+\frac{1}{\lambda_1\delta_\lambda^2}\eta^2(V_{h_n})\Big)}
%+ 2\beta \gamma^{\varpi_n} \Big(1 + \frac{2\gamma^{\varpi}\beta}{1-\gamma^{\varpi}\beta}\Big)\right)\delta(u,V_{h_n}).
\end{eqnarray*}
This is the desired result (\ref{Final_Error_fun_H1}).

From (\ref{Delta_V_h_P_h}), \eqref{Err_Norm_1_Superclose}, (\ref{Error_2}), (\ref{Inequality_11}),  \eqref{FM_Err_fun_H1}, \eqref{FM_Err_fun_L2}
and  (\ref{Final_Error_fun_H1}), $\|u- u_{i,h_n}\|_b$ has following
estimates
\begin{eqnarray*}
&&\|u- u_{i,h_n}\|_b \leq \|u-\bar u_{i,h_n}\|_b + \|\bar u_{i,h_n}-u_{i,h_n}\|_b\nonumber\\
&&\leq 2\Big(1+\frac{1}{\lambda_1\delta_\lambda}\Big)\eta(V_{h_n})\|u- \bar u_{i,h_n}\|_a
+  2\Big(1+\frac{1}{\lambda_1\delta_\lambda}\Big)\eta(V_H)\|\bar u_{i,h_n} - u_{i,h_n}\|_a\nonumber\\
&&\leq 2\Big(1+\frac{1}{\lambda_1\delta_\lambda}\Big)\eta(V_{h_n})
\frac{1}{1-\bar D_\lambda\eta(V_H)}\delta(u,V_{h_n}) \nonumber\\
&&\ \ \ + 2\Big(1+\frac{1}{\lambda_1\delta_\lambda}\Big)\eta(V_H)\frac{2\beta}{1-\bar D_\lambda\eta(V_H)}\gamma^{\varpi_n} \Big(1 + \frac{\gamma^{\varpi}\beta}{1-\gamma^{\varpi}\beta}\Big)\delta(u,V_{h_n})\nonumber\\
&&\leq  2\Big(1+\frac{1}{\lambda_1\delta_\lambda}\Big)\eta(V_H)\frac{1}{1-\bar D_\lambda\eta(V_H)} \left(1+
2\beta \gamma^{\varpi_n}
\Big(1 + \frac{\gamma^{\varpi}\beta}{1-\gamma^{\varpi}\beta}\Big) \right)\delta(u,V_{h_n}).
\end{eqnarray*}
This is the desired result (\ref{Final_Error_fun_L2}).
From (\ref{Rayleigh_quotient_expansion}) and \eqref{Final_Error_fun_H1}, the error estimate for $|\lambda-\lambda_{i,h_n}|$ can be deduced as follows
\begin{eqnarray*}
|\lambda-\lambda_{i,h_n}| \leq \frac{\|u - u_{i,h_n}\|_a^2}{\|u_{i,h_n}\|_b^2}
\leq \lambda_{i,h_n}\|u - u_{i,h_n}\|_a^2.
\end{eqnarray*}
Then  the desired result (\ref{Final_Error_eigen}) is obtained and the proof is complete.
\end{proof}
%------------------------------------------------------------------------------------------------------
From Proposition \ref{propCondition_1} and Theorem \ref{Error_Estimate_Theorem_Final},
it is easy to deduce following explicit error estimates for the
eigenpair approximation $(\lambda_{i,h_n}, u_{i,h_n})$ by Algorithm \ref{Para_Multigrid}.
\begin{corollary}\label{Error_Estimate_Corollary_Final}
After implementing Algorithm \ref{Para_Multigrid}, there exists an eigenpair $(\lambda, u)$ of (\ref{weak_eigenvalue_problem}) such that
the eigenpair approximation $(\lambda_{i,h_n},u_{i,h_n})$ has the following error estimates
for $i=1,\cdots,m$
\begin{eqnarray}
&&\|u-u_{i,h_{n}}\|_a \leq \frac{C}{1-\bar D_\lambda H}\left(1+ 2\beta \gamma^{\varpi_n}\right)\sqrt{\lambda_i}h_n,\label{Final_Error_fun_H1_Corollary}\\
&&\|u-u_{i,h_{n}}\|_b\leq  \frac{CH}{1-\bar D_\lambda H}\left(1+ 2\beta \gamma^{\varpi_n}\right)\sqrt{\lambda_i}h_n,\label{Final_Error_fun_L2_Corollary}\\
&&|\lambda_i -\lambda_{i,h_{n}}| \leq \lambda_i^2 \left(\frac{C}{1-\bar D_\lambda H}\right)^2
\left(1+ 2\beta \gamma^{\varpi_n}\right)^2h_n^2,\label{Final_Error_eigen_Corollary}
\end{eqnarray}
where the constant $C$ depends on $\lambda_1$, spectral gap $\delta_\lambda$, $\gamma$
and $\beta$ but independent of the mesh size $h_n$.
\end{corollary}
%-----------------------------------------------------------------------------------------------------------

\begin{remark}
When $m=1$,  Algorithm \ref{Para_Multigrid} becomes a sequential algorithm.
Even in this case, we can  deal with different eigenpair individually,
which always has a better efficiency than traditional algorithm when the
number of desired eigenpairs is large enough.
\end{remark}

%-------------------------------------------------------------------------------------------------

%-------------------------------------------------------------------------------------------------
\section{Work estimate of parallel augmented subspace method}
Now we turn our attention to the estimate of computational work
for the parallel augmented subspace scheme defined by Algorithm \ref{Para_Multigrid}.

First, we define the dimension of each level of finite element space as $N_k:={\rm dim}V_{h_k}$.
Then the following property holds
\begin{eqnarray}\label{relation_dimension}
N_k\approx\Big(\frac{1}{\beta}\Big)^{d(n-k)}N_n,\ \ \ k=1,2,\cdots, n.
\end{eqnarray}

\begin{theorem}\label{work}
and the work of multigrid Assume that the eigenvalue problem solving in the coarse spaces $V_{H}$ and $V_{h_1}$ need work
$\mathcal{O}(M_H)$ and $\mathcal{O}(M_{h_1})$, respectively,  for the
boundary value problem (\ref{correct_source_exact_para})
in each process is $\mathcal{O}(N_k)$ in the $k$-th level of mesh.
Then the most work involved in each
computing node of {\it Algorithm  \ref{Para_Multigrid} } is
$\mathcal{O}((\varpi_n+\varpi/\beta^d)N_n+ (\varpi\log N_n+\varpi_n)M_H+M_{h_1})$
and the included constant is independent of
the number $m$ of the desired eigenpairs.
Furthermore, the complexity will be $\mathcal{O}((\varpi_n+\varpi/\beta^d)N_n)$
provided $M_H\ll N_k$ and $M_{h_1}\leq N_n$.
\end{theorem}
%---------------------------------------------------------------------------------------------------
\begin{proof}
Let $W_k$ denote the work in each computing node for the correction  step which is defined by
Algorithm \ref{one correction step}
in the $k$-th level of finite element space $V_{h_k}$.
Then from the  definitions of Algorithms \ref{one correction step} and \ref{Para_Multigrid}, we have
\begin{eqnarray}\label{work_k}
W_k=\mathcal{O}(\varpi(N_k+M_H)), \ {\rm for}\  k=2, \cdots, n-1,\ \ {\rm and}\ \
W_n=\mathcal{O}(\varpi_n(N_n+M_H)).
\end{eqnarray}
Iterating (\ref{work_k}) and using the fact (\ref{relation_dimension}), we obtain
\begin{eqnarray}\label{Work_Estimate}
{\rm Total\ Work} &=&\sum_{k=1}^nW_k= \mathcal{O}\left(M_{h_1}+\sum_{k=2}^{n-1}\varpi\big(N_k+M_H\big) + \varpi_n\big(N_n+M_H\big)\right)\nonumber\\
&=&\mathcal{O}\left(M_{h_1}+ (\varpi(n-2)+\varpi_n)M_H+\sum_{k=2}^{n-1}\varpi N_k + \varpi_nN_n\right)\nonumber\\
&=&\mathcal{O}\left(M_{h_1}+(\varpi(n-2)+\varpi_n)M_H+\sum_{k=2}^{n-1}\left(\frac{1}{\beta}\right)^{d(n-k)}\varpi N_n+\varpi_n N_n\right)\nonumber\\
&=&\mathcal{O}\left(\Big(\varpi_n+ \frac{\varpi}{\beta^d}\Big)N_n+ (\varpi\log N_n+\varpi_n)M_H+M_{h_1}\right).
\end{eqnarray}
This is the desired result $\mathcal{O}((\varpi_n+\varpi/\beta^d)N_n+ (\varpi\log N_n+\varpi_n)M_H+M_{h_1})$ and the
one $\mathcal{O}((\varpi_n+\varpi/\beta^d)N_n)$  can be obtained by the conditions $M_H\ll N_n$ and $M_{h_1}\leq N_n$.
\end{proof}
%---------------------------------------------------------------------------------------------------
\begin{remark}
Since there exists no data transfer between different processes, the total computational work
of Algorithm \ref{Para_Multigrid} in each process
is equal to that of one process  for only one eigenpair.

Further, since $\gamma$ has a uniform bound from $1$ ($\gamma < 1$),
then we do not need to do many correction steps in
each level of finite element space.  As in Remark \ref{Remark_Gamma},
we choose $\varpi=1$ for $k=2,\cdots, n-2$ and $\varpi_n$ is dependent on the
algebraic accuracy $\varepsilon$. Then the final computational work in each processor should be $\mathcal{O}(N_n|\log\varepsilon|)$ and
the included constant is independent of the number $m$ of the desired eigenpairs.
\end{remark}
%---------------------------------------------------------------------------------------------------

\section{Numerical results}
In this section, we provide four numerical examples to validate the proposed numerical method in this paper.

\subsection{The model eigenvalue problem}
In this subsection, we use Algorithm \ref{Para_Multigrid} to solve the following model eigenvalue problem:
Find $(\lambda,u)\in\mathcal R\times V$ such that $\|u\|_a=1$ and
\begin{eqnarray}\label{Model_Eigenvalue_Problem}
\left\{
\begin{array}{rcl}
-\Delta u &=&\lambda u,\ \ \ {\rm in}\ \Omega,\\
u&=&0,\ \ \ \ \ {\rm on}\ \partial\Omega,
\end{array}
\right.
\end{eqnarray}
where $\Omega = (0,1)\times (0, 1)\times (0, 1)$.

In this example, we choose $H = 1/16$, $\beta =2$ and $\varpi=\varpi_n=1$. % and $n=5$.
In the first step of one correction step defined by Algorithm \ref{one correction step},
$1$ multigrid step with $2$ Conjugate Gradient (CG) steps for pre- and post-smoothing is adopted to
solve the linear problem (\ref{correct_source_exact_para}).

\revise{First, we investigate the efficiency of the proposed algorithm.  For this aim,
Algorithm \ref{Para_Multigrid} is compared with
the LOBPCG method \cite{Knyazev_Lobpcg,Knyazev2,KnyazevNeymeyr} from the package: Slepc \cite{SLEPC}.
For the sake of fairness, these two methods use the same number of processors and
the linear equations included in LOBPCG method are solved by multigrid iterations.
Furthermore, they use the same convergence criterion which is to be $\|Ax-\lambda Bx\|_2/|\lambda|\leq 1{\tt e}$-$8$
for the algebraic eigenvalue problem: $Ax=\lambda Bx$, where $\|\cdot\|_2$ denotes the $L^2$-norm for vectors.
%\comm{The convergence criterion is set to be $\frac{\|Ax-\lambda Bx\|_2}{|\lambda|}$}.
The numerical comparisons are carried out on LSSC-IV in the State Key Laboratory of Scientific and Engineering
Computing, Chinese Academy of Sciences. Each computing node has two $18$-core Intel Xeon Gold $6140$
processors at $2.3$ GHz and $192$ GB memory.
%For more information, please check \url{http://lsec.cc.ac.cn/chinese/lsec/LSSC-IVintroduction.pdf}.
Here, we use $200$ processors for computing the first $200$ eigenpairs
and  $1000$ processors for the first $1000$ eigenpairs.
Tables \ref{table1} and \ref{table2} show the information of CPU time
for computing the $200$ and $1000$ eigenpairs by Algorithm \ref{Para_Multigrid} and LOBPCG method.
The corresponding memory consumptions are shown in Tables \ref{table3} and \ref{table4}.
From these tables, we can find that Algorithm \ref{Para_Multigrid} has
better efficiency and smaller memory consumption than the LOBPCG method. }

\begin{table}[!hbt]
\begin{center}
\caption{The CPU time of Algorithm \ref{Para_Multigrid} and LOBPCG for
the first $200$ eigenpairs of (\ref{Model_Eigenvalue_Problem}), where
the symbol ``$-$" means the computer runs out of memory.}\label{table1}
\vskip0.2cm
\begin{tabular}{|c|c|c|}\hline
Number of Dofs  & Time of LOBPCG & Time of Algorithm 2 \\ \hline
 274625        &367.75928      &   130.725864  \\ \hline
 2146689       & 5857.01709     &  237.001301  \\ \hline
 16974593      & $-$       &    1107.676507\\ \hline
\end{tabular}
\end{center}
\end{table}
%------------------------------------------
\begin{table}[!hbt]
\begin{center}
\caption{The CPU time of Algorithm \ref{Para_Multigrid} and LOBPCG for
the first $1000$ eigenpairs of (\ref{Model_Eigenvalue_Problem}),
where the symbol $``-"$ means the computer runs out of memory.}\label{table2}
\vskip0.2cm
\begin{tabular}{|c|c|c|}\hline
Number of Dofs  & Time of LOBPCG & Time of Algorithm 2 \\ \hline
 274625        &2826.92015       &  132.50253   \\ \hline
 2146689       & $-$      &    242.77372   \\ \hline
 16974593      & $-$       &   1150.86141\\ \hline
\end{tabular}
\end{center}
\end{table}

\begin{table}[hbt!]
\begin{center}
\caption{The max memory (in $\mathrm{MB}$ ) of Algorithm \ref{Para_Multigrid} and LOBPCG for the first $200$
eigenpairs of (\ref{Model_Eigenvalue_Problem}), where the symbol $``-"$ means the computer runs out of memory.}\label{table3}
\begin{tabular}{|c|c|c|}
\hline Number of Dofs & Memory of LOBPCG & Memory of Algorithm 2 \\
\hline 274625 & 1342 & 353 \\
\hline 2146689 & 6273 & 1554 \\
\hline 16974593 & $-$ & 7251 \\
\hline
\end{tabular}
\end{center}
\end{table}

\begin{table}[hbt!]
\begin{center}
\caption{The max memory (in $\mathrm{MB}$ ) of Algorithm \ref{Para_Multigrid} and LOBPCG for the first $1000$
eigenpairs of (\ref{Model_Eigenvalue_Problem}), where the symbol $``-"$ means the computer runs out of memory.}\label{table4}
\begin{tabular}{|c|c|c|}
\hline Number of Dofs & Memory of LOBPCG & Memory of Algorithm 2 \\
\hline 274625 & 3802 & 366 \\
\hline 2146689 &  $-$ & 1627 \\
\hline 16974593 & $-$ & 7452 \\
\hline
\end{tabular}
\end{center}
\end{table}

%In order to check the parallel property of Algorithm \ref{Para_Multigrid},
%we compute the first $200$ eigenpairs of (\ref{Model_Eigenvalue_Problem}).

Figure \ref{ex1-error-time200} shows the corresponding error estimates of $|\lambda_i-\lambda_{i,h_n}|$
for $i=1, \cdots, 200$ and the CPU time for each eigenpair, respectively.
\begin{figure}[hbt!]
\centering
\includegraphics[width=6.4cm,height=5.4cm]{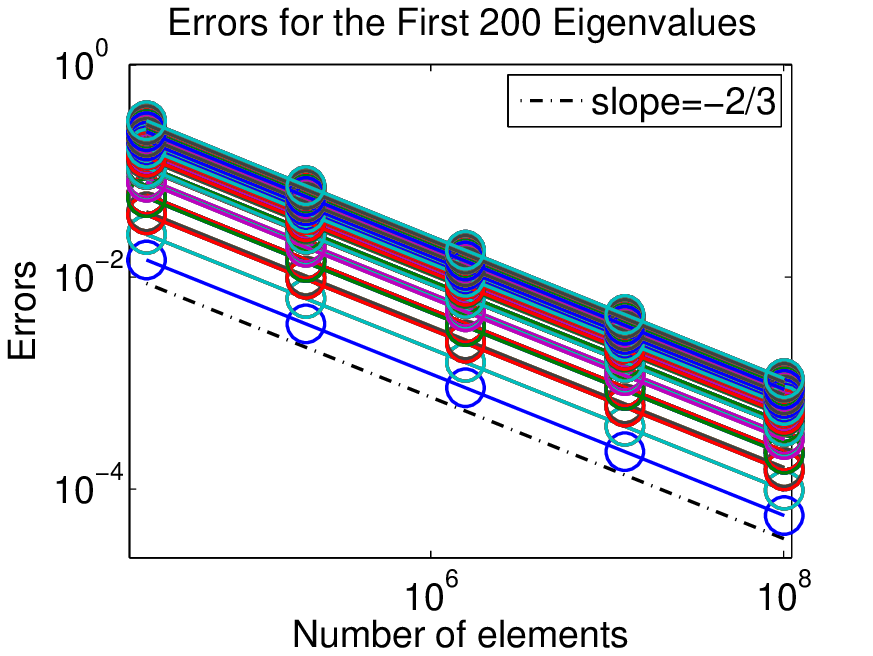}
\includegraphics[width=6.4cm,height=5.4cm]{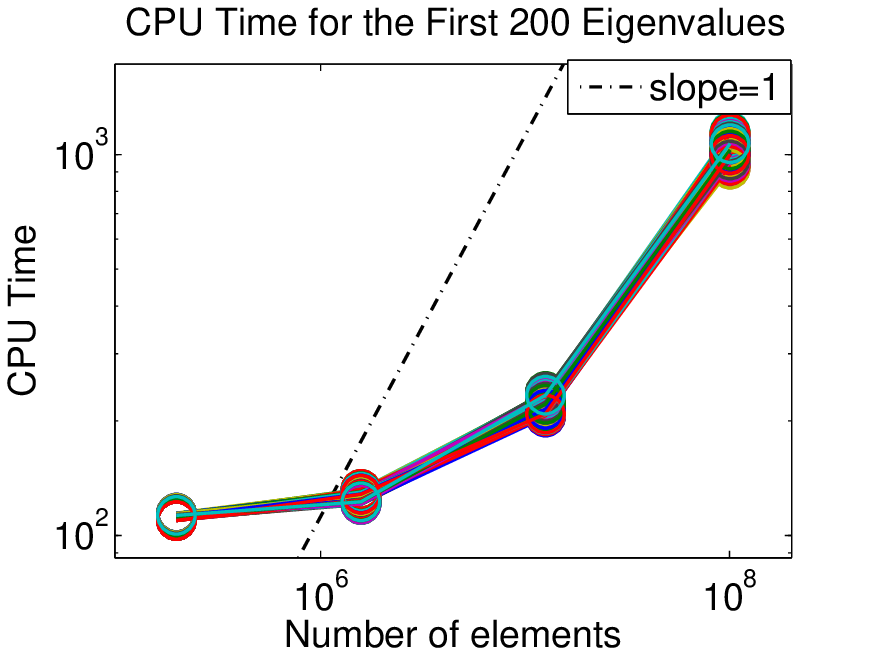}
\caption{\small  The errors and CPU time (in second)
of the parallel augmented subspace method for the
first 200 eigenpairs of Example 1.}\label{ex1-error-time200}
\end{figure}
From Figure \ref{ex1-error-time200}, we can find that
Algorithm \ref{Para_Multigrid} has the optimal error estimate, and needs similar computational work for different eigenpair.
\revise{Here, the optimal error estimate means the concerned eigenvalue approximations have
the errors which are bounded by the discretization errors of the finite element method on the
corresponding level of meshes.}

We also test the algebraic errors $|\bar\lambda_{i,h_n}-\lambda_{i,h_n}|$
between the numerical approximations by Algorithm \ref{Para_Multigrid}
and the exact finite element solutions for the first $20$ eigenvalues
on the finest level of mesh. The corresponding results are presented in
Figure \ref{ex1-error20} which shows that the algebraic accuracy  improves along
with the growth of numbers of correction steps $\varpi_n$.
\begin{figure}[hbt!]
\centering
\includegraphics[width=6.4cm,height=5.4cm]{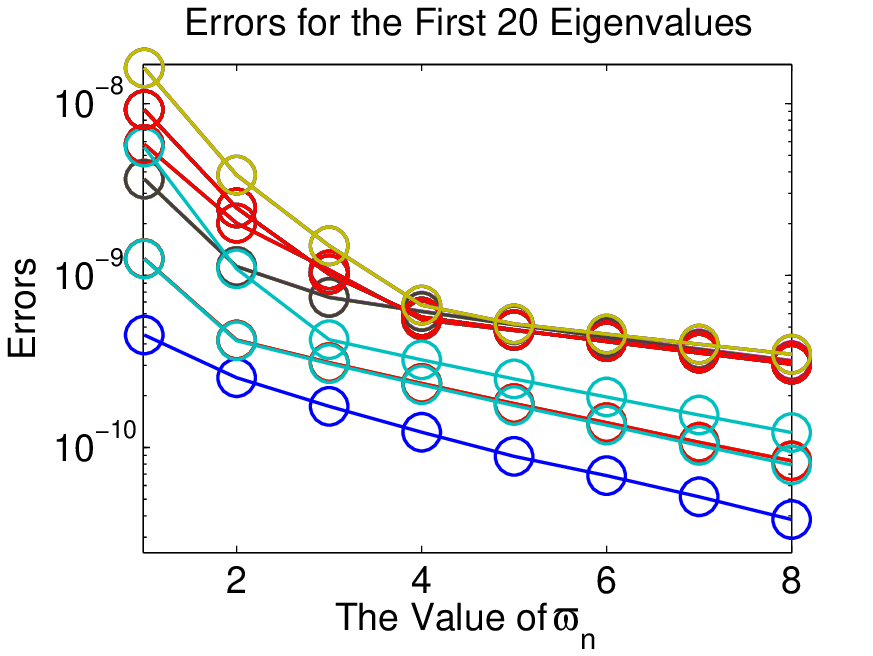}
\caption{\small The algebraic errors for the first
20 eigenvalues  of Example 1 by the parallel augmented subspace method
with different number of correction steps.}\label{ex1-error20}
\end{figure}

The performance of Algorithm \ref{Para_Multigrid} for computing the first $1000$ eigenpairs
is also investigated. Figure \ref{ex1-error-time1000} shows the error estimate and CPU time
for each eigenvalue. %Furthermore,  the corresponding CPU time for each eigenpair
%is also presented in Figure \ref{ex1-error-time1000}.
From Figure \ref{ex1-error-time1000}, we can also find the  parallel method
has optimal convergence order even for the
first $1000$ eigenpairs. These results show the efficiency of Algorithm \ref{Para_Multigrid} and
validity of Theorem \ref{Error_Estimate_Theorem_Final} and
Corollary \ref{Error_Estimate_Corollary_Final}.
\begin{figure}[hbt!]
\centering
\includegraphics[width=6.4cm,height=5.4cm]{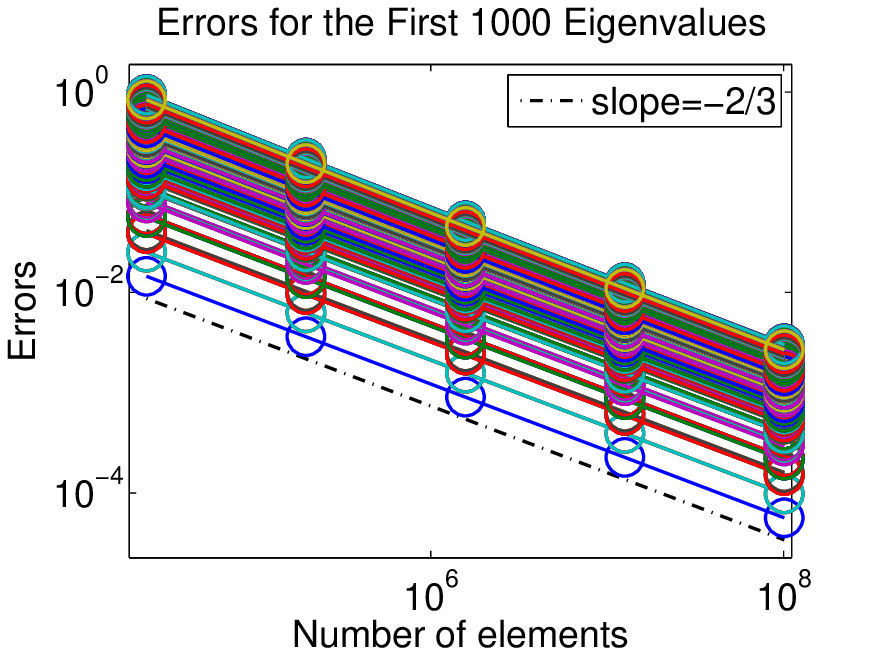}
\includegraphics[width=6.4cm,height=5.4cm]{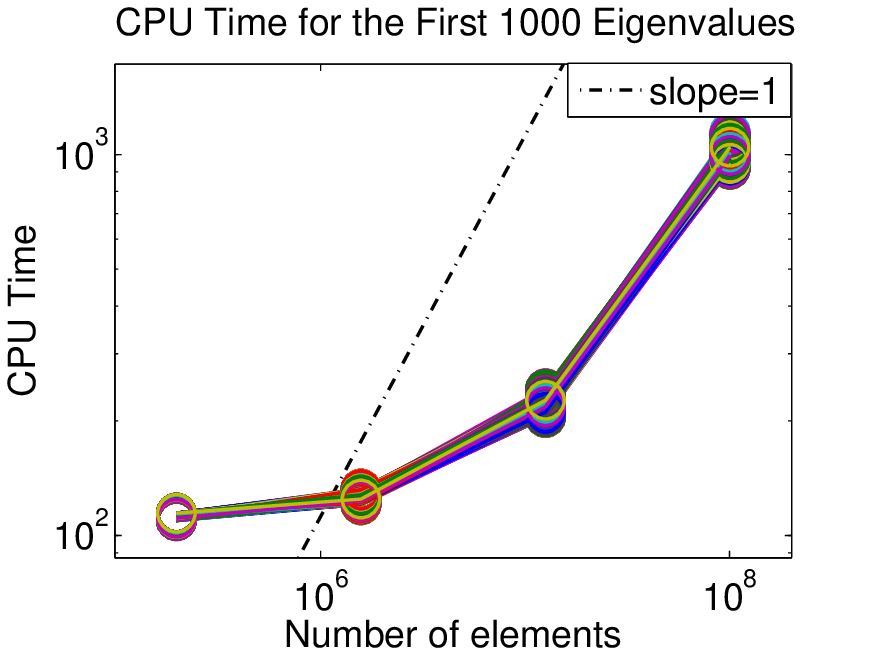}
\caption{\small The errors and CPU time (in second) of the
parallel augmented subspace method for the first 1000 eigenpairs
of Example 1.}\label{ex1-error-time1000}
\end{figure}
%\begin{figure}[hbt!]
%\centering
%\includegraphics[width=6.4cm,height=5.4cm]{ex1-time.eps}
%\caption{\small The computational time of the parallel augmented
%subspace method for the first 1000 eigenvalues  of Example 1.}\label{ex1-time}
%\end{figure}

In order to check the orthogonality of approximate eigenfunctions by Algorithm \ref{Para_Multigrid},
we investigate the inner products of eigenfunctions corresponding to different eigenvalues.
%which is stated in Remark \ref{Remark_Orthogonal}. %is also tested.
%For this aim,
We compute inner products for the first $100$ approximate eigenfunctions
on the finest level of mesh  by Algorithm \ref{Para_Multigrid}.
Figure \ref{ex1-inner} shows the biggest values of inner product of
eigenfunctions according to different eigenvalues along with the growth of correction steps on the finest level of mesh.
The results in Figure \ref{ex1-inner} show
that Algorithm \ref{Para_Multigrid} can keep the orthogonality
when the algebraic accuracy is small enough, which validates Remark \ref{Remark_Orthogonal}.
\begin{figure}[hbt!]
\centering
\includegraphics[width=6.4cm,height=5.4cm]{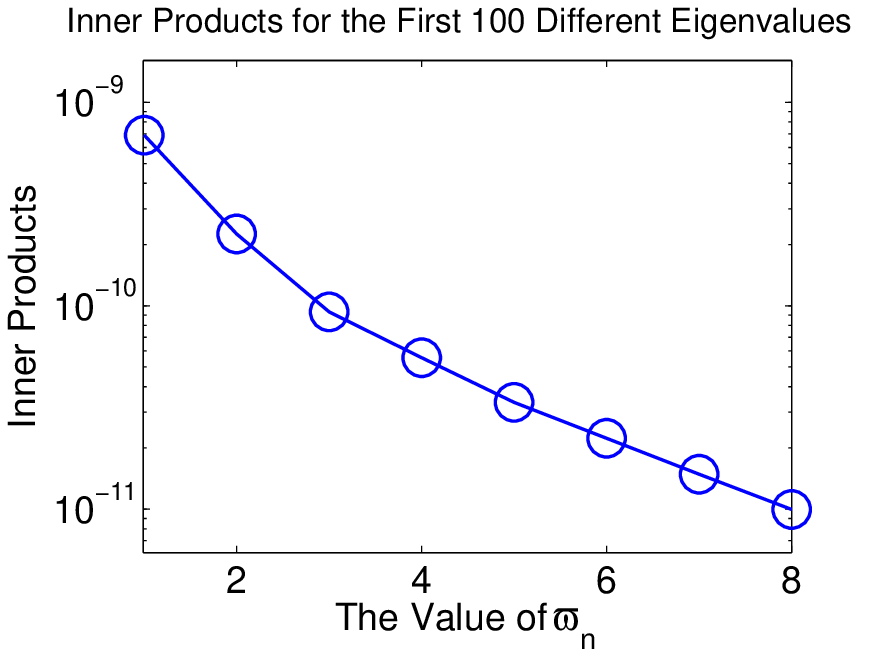}
\caption{\small  The inner products of the eigenfunctions corresponding to the first 100 different
eigenvalues of Example 1.}\label{ex1-inner}
\end{figure}

\subsection{A more general eigenvalue problem}
In this example, we consider the following second order elliptic
eigenvalue problem:
Find $(\lambda,u)\in\mathcal R\times V$ such that $\|u\|_a=1$ and
\begin{equation}\label{general}
\left\{
\begin{array}{rcl}
-\nabla\cdot(\mathcal{A}\nabla u) +\varphi u&=&\lambda u\ \ \ {\rm in}\ \Omega,\\
u&=&0\ \ \ \ \ {\rm on}\ \partial\Omega,\\
\|u\|_{0,\Omega}&=&1,
\end{array}
\right.
\end{equation}
where
\begin{equation*}
\mathcal{A}=\left(
\begin{array}{ccc}
$$1+(x_1-\frac{1}{2})^2$$&$$(x_1-\frac{1}{2})(x_2-\frac{1}{2})$$&$$(x_1-\frac{1}{2})(x_3-\frac{1}{2})$$\\
$$(x_1-\frac{1}{2})(x_2-\frac{1}{2})$$&$$1+(x_2-\frac{1}{2})^2$$&$$(x_2-\frac{1}{2})(x_3-\frac{1}{2})$$\\
$$(x_1-\frac{1}{2})(x_3-\frac{1}{2})$$&$$(x_2-\frac{1}{2})(x_3-\frac{1}{2})$$&$$1+(x_3-\frac{1}{2})^2$$
\end{array}
\right),
\end{equation*}
$\varphi=e^{(x_1-\frac{1}{2})(x_2-\frac{1}{2})(x_3-\frac{1}{2})}$ and
$\Omega = (0,1)\times (0, 1)\times (0, 1)$.

In order to check the parallel property of Algorithm \ref{Para_Multigrid},
we compute the first $200$ eigenpairs of (\ref{general}).
Here, we choose $H = 1/16$, $\beta =2$, $\varpi=\varpi_n=1$ and $n=5$.
In the first step of one correction step defined by Algorithm \ref{one correction step},
$1$ multigrid step with $2$ CG steps for pre- and post-smoothing is adopted to
solve the linear problem (\ref{correct_source_exact_para}).
Since the exact solutions are not known, the adequate accurate
approximations are chosen as the exact solutions for our numerical test.
Figure \ref{ex2-error-time200} shows the corresponding error estimates of
$|\lambda_i-\lambda_{i,h_n}|$ for $i=1, \cdots, 200$
and the CPU time for each eigenpair, respectively.
From Figure \ref{ex2-error-time200}, we  can find that
Algorithm \ref{Para_Multigrid} has the optimal error estimate,
and needs similar computational work for computing different eigenpair.
\begin{figure}[hbt!]
\centering
\includegraphics[width=6.4cm,height=5.4cm]{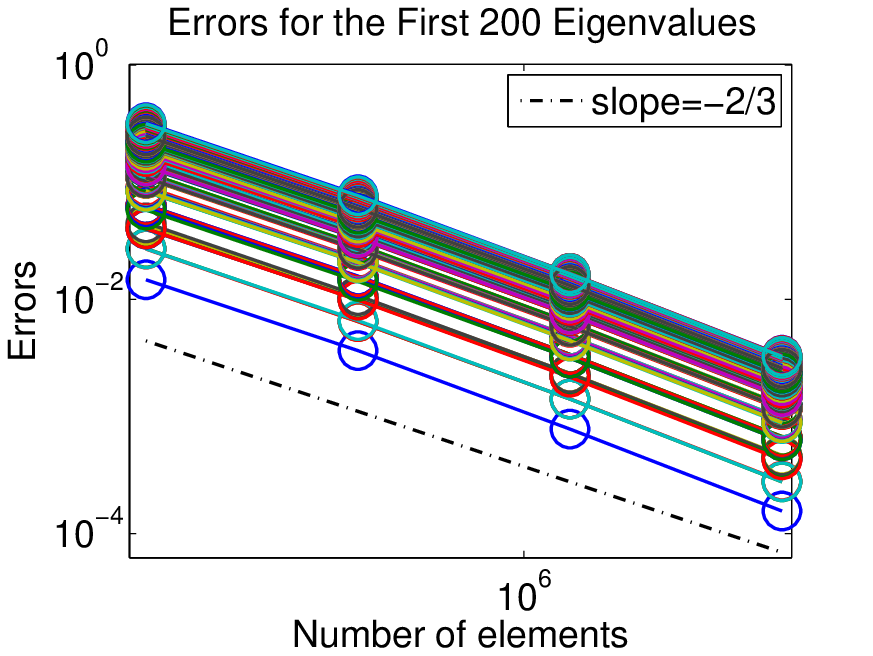}
\includegraphics[width=6.4cm,height=5.4cm]{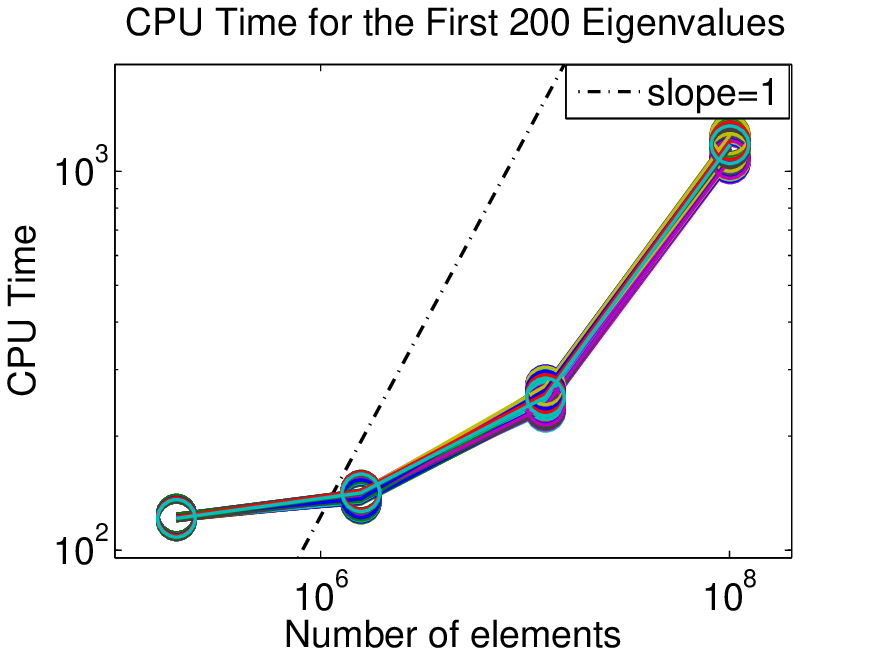}
\caption{\small  The errors and CPU time (in second) of the parallel augmented subspace
method for the first 200 eigenpairs  of Example 2.}\label{ex2-error-time200}
\end{figure}

In this example, we also test the algebraic error $|\bar\lambda_{i,h_n}-\lambda_{i,h_n}|$
between the numerical approximations by Algorithm \ref{Para_Multigrid} and the exact finite element
solutions for the first $20$ eigenvalues along with the growth of the number of correction steps.
The corresponding results are presented in Figure \ref{ex2-error20},  which shows
that the algebraic accuracy improves along with the growth of $\varpi_n$.
\begin{figure}[hbt!]
\centering
\includegraphics[width=6.4cm,height=5.4cm]{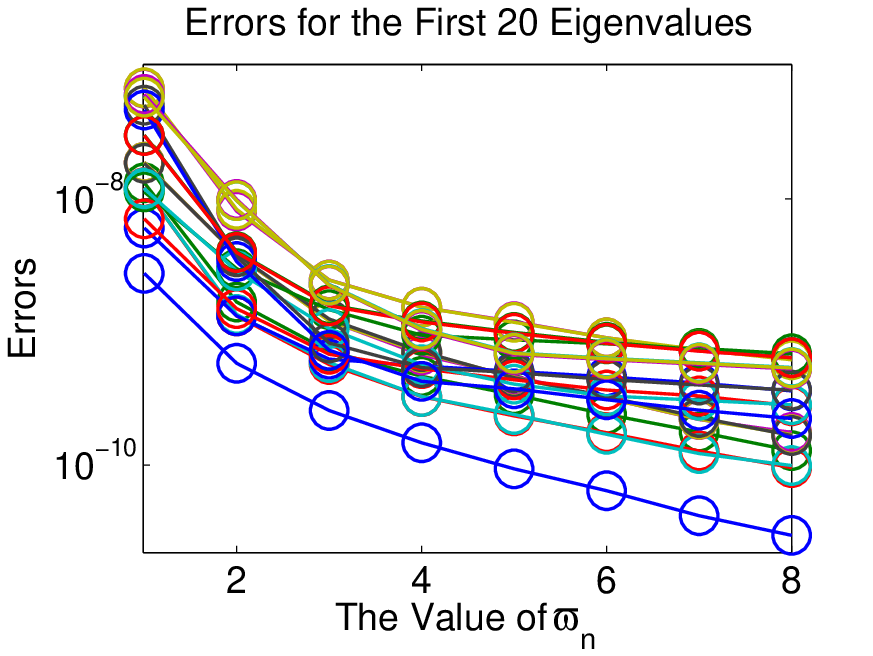}
\caption{\small The algebraic errors for the first 20 eigenvalues  of Example 2 by the parallel augmented subspace
method with different number of correction step times.}\label{ex2-error20}
\end{figure}

Then, we compute the first $1000$ eigenpairs. The corresponding error estimates
for the approximate eigenvalues and CPU time for each eigenpair are shown in
Figure \ref{ex2-error-time1000}. From Figure \ref{ex2-error-time1000},
we can also find that the  parallel method
can arrive the theoretical convergence order for the first $1000$ eigenpairs.
These results show the efficiency of Algorithm \ref{Para_Multigrid} and the validity
of Theorem \ref{Error_Estimate_Theorem_Final} and Corollary \ref{Error_Estimate_Corollary_Final}.
\begin{figure}[hbt!]
\centering
\includegraphics[width=6.4cm,height=5.4cm]{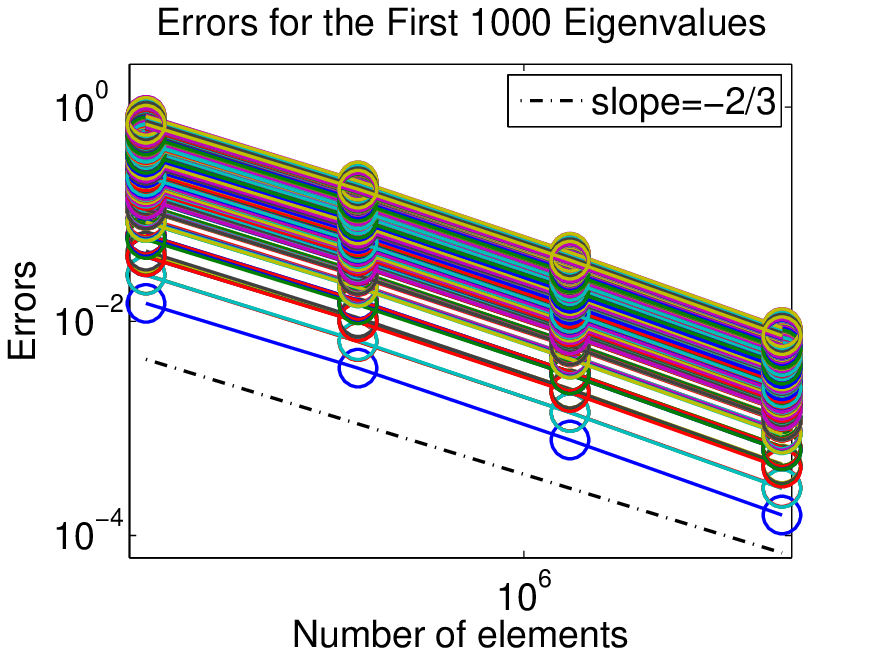}
\includegraphics[width=6.4cm,height=5.4cm]{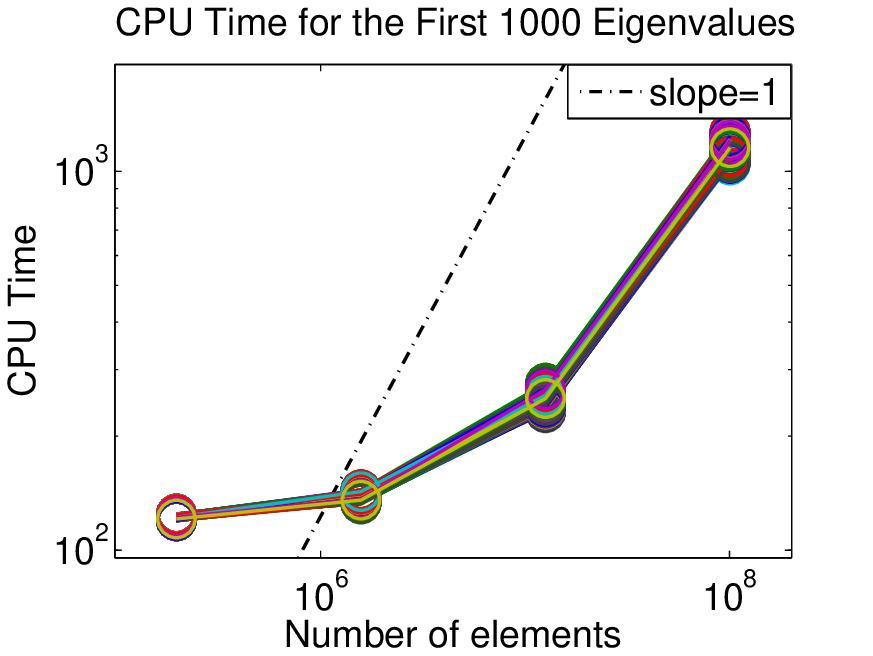}
\caption{\small The errors and CPU time (in second) of the parallel augmented subspace
method for the first 1000 eigenpairs  of Example 2.}\label{ex2-error-time1000}
\end{figure}

The orthogonality of approximate eigenfunctions by Algorithm \ref{Para_Multigrid} is also tested.
In Figure \ref{ex2-inner}, we also show the biggest values of inner product for the
first $100$ approximate eigenfunctions according to different eigenvalues  on the finest level of mesh.
Form Figure \ref{ex2-inner}, it can be found that Algorithm \ref{Para_Multigrid} can keep the orthogonality
when the algebraic accuracy is small enough.
\begin{figure}[hbt!]
\centering
\includegraphics[width=6.4cm,height=5.4cm]{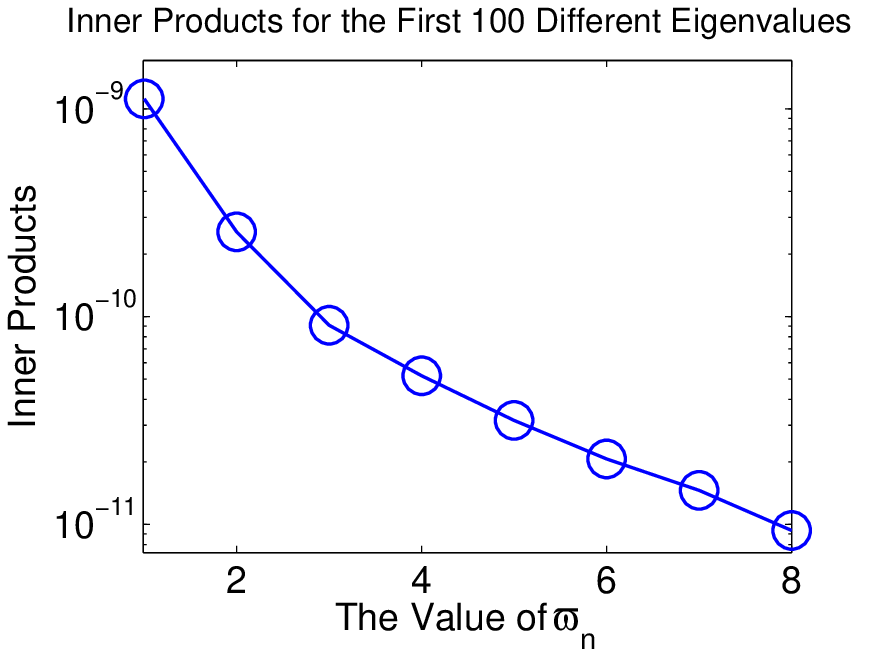}
\caption{\small  The inner products of the eigenfunctions corresponding to the first 100
different eigenvalues of Example 2.}\label{ex2-inner}
\end{figure}

\subsection{Adaptive finite element method}
In this example, we consider the following eigenvalue problem
(see \cite{Greiner}): Find $(\lambda,u)\in\mathcal R\times V$ such that $\|u\|_a=1$ and
\begin{equation}\label{afem_Eigenvalue_Problem}
-\frac{1}{2}\Delta u +\frac{1}{2}|x|^2u=\lambda u\ \ \ {\rm in}\ \Omega,
\end{equation}
where $\Omega= \mathcal R^3$ and $|x|=\sqrt{x_1^2+x_2^2+x_3^2}$.
The eigenvalues of (\ref{afem_Eigenvalue_Problem}) are
$$\lambda_{i,j,k}= i+j+k+\frac{3}{2},$$
where $i, j, k$ denote the integral numbers and $i,j,k \geq 0$.
Since the eigenfunctions are exponential decay, we set $\Omega=(-4, 4)^3$
and the boundary condition $u=0\ {\rm on}\ \partial\Omega$ in our computation for simplicity.
Since the exact eigenfunction with singularities is expected, the adaptive refinement
is adopted to couple with Algorithm \ref{Para_Multigrid} (cf. \cite{HongXieXu}).

In order to check the parallel property of Algorithm \ref{Para_Multigrid},
we compute the first $200$ eigenpairs of (\ref{afem_Eigenvalue_Problem}).
In this example, we choose $H = 1/4$, $\varpi=\varpi_n=1$.
In the first step of one correction step defined by Algorithm \ref{one correction step},
$1$ multigrid step with $2$ CG steps for pre- and post-smoothing is adopted to
solve the linear problem (\ref{correct_source_exact_para}).
Figure \ref{ex3-error-time200} shows the corresponding error estimates of $|\lambda_i-\lambda_{i,h_n}|$ for $i=1, \cdots, 200$ and the
CPU time for each eigenpair. From Figure \ref{ex3-error-time200}, we can find that Algorithm \ref{Para_Multigrid} has the optimal error estimates and
needs similar computational work for different eigenvalue even on the adaptively refined meshes. These results show that Algorithm \ref{Para_Multigrid}
can be coupled with the adaptive refinement technique.
\begin{figure}[hbt!]
\centering
\includegraphics[width=6.4cm,height=5.4cm]{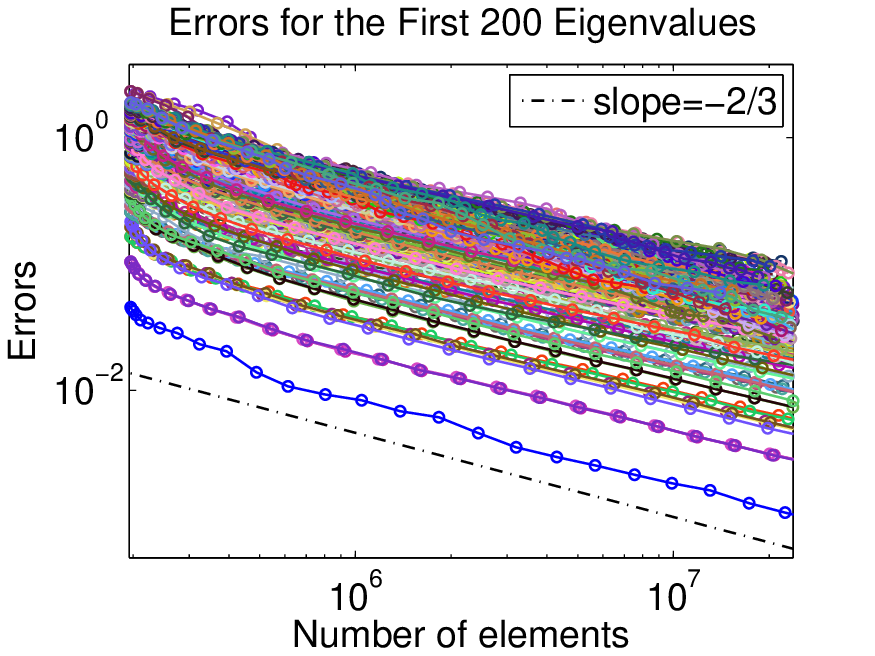}
\includegraphics[width=6.4cm,height=5.4cm]{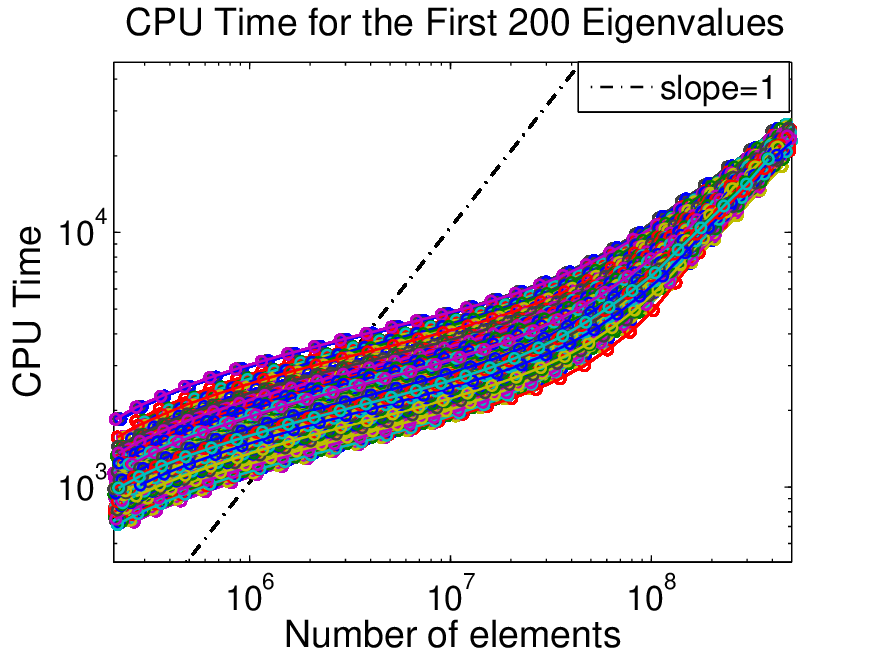}
\caption{\small  The errors and CPU time (in second) of the parallel augmented subspace
method for the first 200 eigenvalues  of Example 3.}
\label{ex3-error-time200}
\end{figure}

\subsection{Adaptive finite element method for Hydrogen atom}
In order to show the potential for electrical structure simulation, in the last example,
we consider the following model for Hydrogen atom:
Find $(\lambda,u)\in\mathcal R\times V$ such that $\|u\|_a=1$ and
\begin{equation}\label{hy}
-\frac{1}{2}\Delta u-\frac{1}{|x|}u=\lambda u, \quad {\rm in} \  \Omega,
\end{equation}
where $\Omega=\mathcal R^3$. The eigenvalues of (\ref{hy}) are $\lambda_h=-\frac{1}{2n^2}$ with multiplicity $n^2$
for any positive integer $n$. Along with the growths of $n$, it is easy to find that the spectral gap becomes
small and the multiplicity large which improve the difficulty for solving the eigenvalue problem. The aim of
this example is to show the parallel augmented subspace method in this paper can also compute the cluster eigenvalues and their eigenfunctions.
Since the eigenfunction is exponential decay, we also set $\Omega = (-4,4)^3$
and the boundary condition $u=0$ on $\partial\Omega$ in our computation.
Here the adaptive refinement is also adopted to couple with Algorithm \ref{Para_Multigrid}.

In order to check the parallel property of Algorithm \ref{Para_Multigrid},
we compute the first $200$ eigenpairs of (\ref{hy}).
In this example, we choose $H = 1/4$, $\varpi=\varpi_n=1$.
In the first step of one correction step defined by Algorithm \ref{one correction step},
$1$ multigrid step with $2$ CG steps for pre- and post-smoothing is adopted to
solve the linear problem (\ref{correct_source_exact_para}).
Figure \ref{ex4-error-time200} shows the corresponding error estimates of
$|\lambda_i-\lambda_{i,h_n}|$ for $i=1, \cdots, 200$ and the
CPU time for each eigenpair.
From Figure \ref{ex4-error-time200}, we can also find that
Algorithm \ref{Para_Multigrid} has the optimal error estimate and
needs  similar computational work for different eigenvalue.
%These results also shows that Algorithm \ref{Para_Multigrid} can be coupled with the adaptive refinement technique.

\begin{figure}[hbt!]
\centering
\includegraphics[width=6.4cm,height=5.4cm]{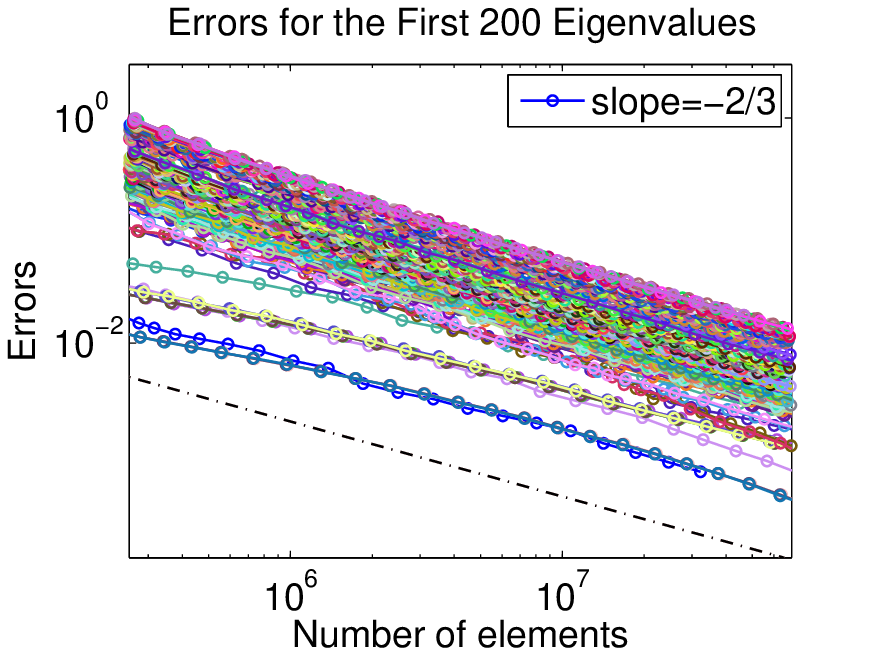}
\includegraphics[width=6.4cm,height=5.4cm]{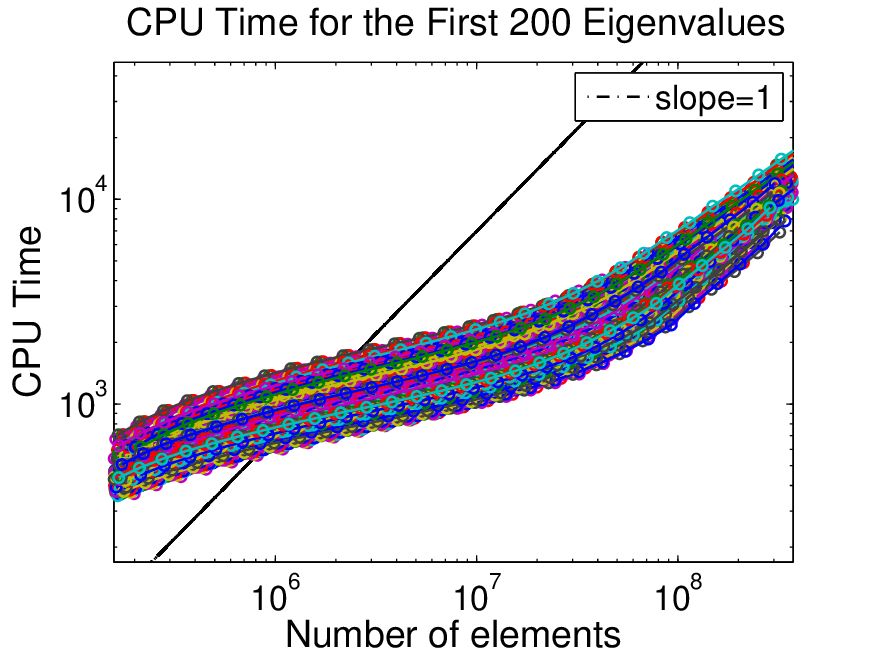}
\caption{\small  The errors and CPU time (in second) of the parallel augmented subspace
method for the first 200 eigenvalues  of Example 4.}
\label{ex4-error-time200}
\end{figure}

\section{Concluding remarks}
In this paper, we propose a parallel augmented subspace scheme for eigenvalue problems
by using the coarse space from the multigrid method. In this numerical method,
solving the eigenvalue problem in the finest space is decomposed into solving
the standard linear boundary value problems and very low dimensional eigenvalue problems.
Furthermore, for different eigenvalue, the corresponding boundary value problem
and low dimensional eigenvalue problem can be solved in the parallel way
since they are independent of each other and there exists no data exchanging.
This property means that we do not need to do the orthogonalization in the highest
dimensional space and the the efficiency and scalability can be improved obviously.

The method in this paper can be applied to other types of linear and nonlinear eigenvalue problems
such as biharmonic eigenvalue problem, Steklov eigenvalue problem, Kohn-Sham equation. These will be
our future work.

%---------------------------------------------------------------------------------------------------


\begin{thebibliography}{17}
\bibitem{Adams}
R. A. Adams, {\em Sobolev Spaces}, Academic Press, New York, 1975.


\bibitem{BabuskaOsborn_1989}
I. Babu\v{s}ka and J. Osborn, {\em Finite element-Galerkin
approximation of the eigenvalues and eigenvectors of selfadjoint
problems}, Math. Comp., 52 (1989), 275--297.

\bibitem{BabuskaOsborn_Book}
I. Babu\v{s}ka and J. Osborn, {\em Eigenvalue Problems}, In Handbook of
Numerical Analysis, Vol. II, (Eds. P. G. Lions and Ciarlet P.G.),
Finite Element Methods (Part 1), North-Holland, Amsterdam, 641--787,
1991.


\bibitem{Bai}
Z. Bai, J. Demmel, J. Dongarra, A. Ruhe, and H. van der Vorst, editors. {\em  Templates
for the Solution of Agebraic Eigenvalue Problems: A Practical Guide},
Society for Industrial and Applied Math., Philadelphia, 2000.


\bibitem{BankDupont}
R. E. Bank and T. Dupont, {\em An optimal order process for solving finite element equations},
Math. Comp., 36 (1981), 35--51.

\bibitem{Bramble}
J. H. Bramble, {\em Multigrid Methods}, Pitman Research Notes in Mathematics, V. 294, John Wiley and Sons, 1993.


\bibitem{BramblePasciak}
J. H. Bramble and J. E. Pasciak, {\em New convergence estimates for multigrid algorithms},
Math. Comp., 49 (1987), 311--329.

\bibitem{BramblePasciakKnyazev}
J. H. Bramble, J. Pasciak, and A. Knyazev, {\em A subspace preconditioning algorithm for
eigenvector/eigenvalue computation}, Advances in Computational Mathematics, 6(1) (1996), 159--189.

\bibitem{BrambleZhang}
J. H. Bramble and X. Zhang, {\em The Analysis of Multigrid Methods}, Handbook of Numerical Analysis,
Vol. VII, P. G. Ciarlet and J. L. Lions, eds., Elsevier Science, 173--415, 2000.

\bibitem{BrennerScott}
S. Brenner and L. Scott, {\em The Mathematical Theory of Finite Element
Methods}, New York: Springer-Verlag, 1994.

\bibitem{Chatelin}
F. Chatelin, {\em Spectral Approximation of Linear Operators}, Academic
Press Inc, New York, 1983.


\bibitem{full}
H. Chen, H. Xie and F. Xu, {\em A full multigrid method for eigenvalue problems},
J. Comput. Phys., 322 (2016), 747--759.

\bibitem{Ciarlet}
P. G. Ciarlet, {\em The finite Element Method for Elliptic Problem},
North-holland Amsterdam, 1978.

\bibitem{PINVIT}
E. G. D'yakonov and M. Yu. Orekhov, {\em Minimization of the computational labor in determining
the first eigenvalues of differential operators}, Math. Notes, 27 (1980), 382--391.

\bibitem{Greiner}
W. Greiner, {\em Quantum Mechanics: An Introduction}, 3rd edn, Springer, Berlin, 1994.

\bibitem{Hackbusch}
W. Hackbusch, {\em On the computation of approximate eigenvalues and eigenfunctions of
elliptic operators by means of
a multi-grid method}, SIAM J. Numer. Anal., 16(2) (1979), 201--215.

\bibitem{Hackbusch_Book}
W. Hackbusch, {\em Multi-grid Methods and Applications}, Springer-Verlag, Berlin, 1985.

\bibitem{SLEPC}
V. Hernandez, J. E. Roman and V. Vidal, {\em  SLEPc: A scalable and flexible toolkit for the
solution of eigenvalue problems}, ACM Trans. Math. Software, 31(3) (2005), 351--362.

\bibitem{HongXieXu}
Q. Hong, H. Xie and F. Xu, {\em A multilevel correction type of adaptive finite element method for eigenvalue problems},
SIAM J. Sci. Comput.,, 40(6) (2018), A4208--A4235.


\bibitem{Knyazev}
A. Knyazev, {\em Preconditioned eigensolvers-an oxymoron}?
Electronic Transactions on Numerical Analysis, 7 (1998), 104--123.


\bibitem{Knyazev_Lobpcg}
A. Knyazev, {\em Toward the optimal preconditioned eigensolver: Locally optimal block
preconditioned conjugate gradient method}, SIAM Journal on Scientific Computing, 23(2) (2001), 517--541.

\bibitem{Knyazev2}
A. V. Knyazev, M. E. Argentati, I. Lashuk and E. E. Ovtchinnikov, {\em  Block locally optimal preconditioned
eigenvalue xolvers (BLOPEX) in hypre and PETSc},
 SIAM J. Sci. Comput., 29(5) (2007), 2224--2239.


\bibitem{KnyazevNeymeyr}
A. Knyazev and K. Neymeyr, {\em Efficient solution of symmetric eigenvalue problems
using multigrid preconditioners in the locally optimal block conjugate gradient method},
Electronic Transactions on Numerical Analysis., 15 (2003), 38--55.

\bibitem{LinXie_MultiLevel}
Q. Lin and H. Xie, {\em A multi-level correction scheme for eigenvalue problems},
Math. Comp., 84 (2015), 71--88.


\bibitem{McCormick}
S. F. McCormick, ed., {\em Multigrid Methods}. SIAM Frontiers in Applied Matmematics 3.
 Society for Industrial and Applied Mathematics, Philadelphia, 1987.


\bibitem{Saad1}
Y.~Saad, {\em Numerical Methods For Large Eigenvalue Problems}, Society for Industrial and Applied Mathematics, 2011.

\bibitem{ScottZhang}
L. R. Scott and S. Zhang, {\em Higher dimensional non-nested multigrid methods},
Math. Comp., 58 (1992), 457--466.

\bibitem{Shaidurov}
V. V. Shaidurov, {\em Multigrid methods for finite element}, Kluwer
Academic Publics, Netherlands, 1995.


\bibitem{Sorensen}
D. Sorensen, {\em Implicitly Restarted Arnoldi/Lanczos Methods for Large Scale
Eigen value Calculations}, Springer Netherlands, 1997.



\bibitem{StrangFix}
G. Strang and G. J. Fix, {\em An Analysis of the Finite Element Method},
Prentice-Hall, Eiglewood Cliffs, NJ, 1973.

\bibitem{ToselliWidlund}
A. Toselli and O. Widlund, {\em Domain Decomposition Methods: Algorithm and Theory},
Springer-Verlag, Berlin Heidelberg, 2005.

\bibitem{Xie_IMA}
H. Xie, {\em A type of multilevel method for the Steklov eigenvalue problem},
IMA J. Numer. Anal., 34 (2014), 592--608.

\bibitem{Xie_JCP}
H. Xie, {\em A multigrid method for eigenvalue problem},
J. Comput. Phys., 274 (2014),  550--561.


%\bibitem{XieZhangOwhadi}
%H. Xie, L. Zhang and H. Owhadi,
%{\em Fast eigenpairs computation with operator adapted wavelets and hierarchical subspace correction},
%arXiv:1806.00565, 2 June, 2018.

\bibitem{XieZhangOwhadi}
H. Xie, L. Zhang and H. Owhadi, {\em Fast eigenvalue computation with operator adapted wavelets and hierarchical subspace correction},
SIAM J. Numer. Anal, 57(6) (2019), 2519--2550.




\bibitem{Xu}
J. Xu, {\em Iterative methods by space decomposition and subspace
correction}, SIAM Review, 34(4) (1992), 581--613.

\bibitem{Xu_Two_Grid}
J. Xu, {\em A new class of iterative methods for nonselfadjoint or indefinite problems},
SIAM J. Numer. Anal., 29 (1992), 303--319.

\end{thebibliography}
\end{document}